\documentclass[a4paper,12pt]{article}
\usepackage[top=4cm, bottom=4cm, left=3cm, right=3cm]{geometry}
\usepackage[utf8]{inputenc}
\usepackage{amsmath}
\usepackage[rgb,dvipsnames]{xcolor}
\usepackage{tikz}
\usetikzlibrary {arrows.meta,bending,positioning} 
\usepackage{standalone}
\usepackage{bm}
\usepackage{bbm}
\usepackage{graphicx}
\usepackage{amssymb}
\usepackage{amsthm}
\usepackage{mathrsfs}  
\usepackage{mathtools}
\usepackage{hyperref}
\usepackage[lofdepth,lotdepth]{subfig}
\usepackage[font=small,labelfont=bf]{caption}
\usepackage{authblk}

\usepackage[nottoc, notlof, notlot, numbib]{tocbibind}
\usepackage[style=alphabetic,giveninits,doi=false,isbn=false,url=false,eprint=false,clearlang=true,sorting=nyt,maxbibnames=99]{biblatex}
\newcommand{\A}{\mathcal{A}}
\newcommand{\B}{\mathcal{B}}
\newcommand{\C}{\mathcal{C}}
\newcommand{\E}{\mathbb{E}}

\newcommand{\M}{\mathcal{M}}
\newcommand{\R}{\mathbb{R}}
\newcommand{\e}{\mathrm{e}}
\newcommand{\X}{\mathbb{X}}
\newcommand{\dd}{\mathrm{d}}
\renewcommand{\d}{\mathrm{d}}
\newcommand{\1}{\mathbbm{1}}
\newcommand{\pb}{\mathcal{P}}
\newcommand{\ps}{\mathcal{P}}
\newcommand{\W}{\mathcal{W}}

\newtheorem{theorem}{Theorem}[section]
\newtheorem{proposition}[theorem]{Proposition}
\newtheorem{lemma}[theorem]{Lemma}
\newtheorem{corollary}[theorem]{Corollary}

\theoremstyle{definition}
\newtheorem{definition}[theorem]{Definition}

\theoremstyle{remark}
\newtheorem{remark}[theorem]{Remark}

\title{\bf Mean-field limits à la Tanaka and large deviations for particle systems with network interactions}
\author[1]{Louis-Pierre \textsc{Chaintron}}
\author[2]{Antoine \textsc{Diez}}
\date{}

\affil[1]{\small
École Polytechnique Fédédrale de Lausanne (EPFL), Institute of Mathematics, 1015 Lausanne, Switzerland.
}
\affil[2]{\small
Mathematical Application Research Team, Division of Applied Mathematical Science, 
RIKEN Center for Interdisciplinary Theoretical and Mathematical Sciences (iTHEMS), 
RIKEN iTHEMS Wako Saitama 351-0198, Japan
}

\begin{document}

\maketitle

\begin{abstract}

This article proposes a unified framework to study non-exchangeable mean-field particle systems with some general interaction mechanisms. The starting point is a fixed-point formulation of particle systems originally due to Tanaka that allows us to prove mean-field limit and large deviation results in an abstract setting. While it has been recently shown that such formulation encompasses a large class of exchangeable particle systems, we propose here a setting for the non-exchangeable case, including the case of adaptive interaction networks. We introduce sufficient conditions on the network structure that imply the mean-field limit and a new large deviations principle for the interaction measure. Finally, we formally highlight important models for which it is possible to derive a closed PDE characterization of the limit. 

\end{abstract}

\section{Introduction}

In his seminal article \cite{tanaka_limit_1982}, Tanaka considers the by-now classical particle system $\vec{X}^N := ( X^{i,N} )_{1 \leq i \leq N}$ of $N$ interacting diffusion processes $X^{i,N} := (X^{i,N}_t)_{0 \leq t \leq T}$ given by
\begin{equation}\label{eq:mckeanvlasov}
\dd X^{i,N}_t = \frac{1}{N}\sum_{j=1}^N K(X^{i,N}_t,X^{j,N}_t) \dd t + \dd B^i_t, \quad i\in\{1,\ldots,N\},
\end{equation}
where the $B^i := (B^i_t)_{0 \leq t \leq T}$, $i \geq 1$, are independent Brownian motions and $K:\R^d\times\R^d\to\R^d$ is a sufficiently regular interaction function. This system has been originally introduced and studied by McKean \cite{mckean_propagation_1969} following Kac's program \cite{kac_foundations_1956} in mathematical kinetic theory. Since then, McKean's model has become the building block of countless variations and applications. The objective was the rigorous derivation of statistical descriptions of the behavior of this high-dimensional SDE system as the number of particles tends to infinity. Such approach, already prefigured by Boltzmann, reduces the complexity of large stochastic systems into the study of low-dimensional (nonlinear) PDEs, but raises a number of mathematical challenges. We refer to the review articles \cite{jabin_mean_2017,chaintron_propagation_2022,chaintron_propagation_2022a,fathi_recent_2023} and to the classical lecture notes \cite{sznitman_topics_1991,meleard_asymptotic_1996} for more details and an exposition of the numerous applications.

It has been known for a long time that the problem can be essentially reduced to the question of the convergence as $N\to+\infty$ of the sequence of random (pathwise) empirical measures
\[ \pi ( \vec{X}^N) := \frac{1}{N}\sum_{i=1}^N \delta_{X^{i,N}} \in \ps ( \mathcal{C}([0,T],\R^d) ).\]
When the particles are initially independent and identically distributed, the convergence of the time marginal at time $t=0$ is a consequence of the law of large numbers. The so-called \emph{propagation of chaos} result \cite{kac_foundations_1956} shows that despite the interactions, independence is asymptotically preserved at any further time in the many particle limit. The sequence of empirical measures can thus be shown to converge towards the law of a typical particle, which is often characterized as the solution of a PDE or a martingale problem. 

As a first informal description, Tanaka's approach developed in \cite[Section~2.1]{tanaka_limit_1982} is also based on this fundamental idea that any property that holds true for the initial or input data should propagate to the particle system. The main difference is that Tanaka includes not only the initial conditions but also all the necessary information in the construction of the particle dynamics, in the present case, the Brownian motions $B^i$'s. The central object in Tanaka's approach is then a map $\mathbb{X}^{\alpha_N}:\C_{\R^d} \to \C_{\R^d}$ defined on the space $\C_{\R^d} := \mathcal{C}([0,T],\R^d)$ of $\R^d$-valued continuous processes over a fixed time interval $[0,T]$ that transform this \emph{input data} into the McKean-Vlasov system \eqref{eq:mckeanvlasov}. This map is defined by the following fixed-point equation (posed in a suitable functional space):
\[\mathbb{X}^{\alpha_N}_t(\omega) = \int_0^t \int_{\C_{\R^d}} K\big(\mathbb{X}^{\alpha_N}_s(\omega),\mathbb{X}^{\alpha_N}_s(\omega')\big)\alpha_N(\dd\omega') \dd s + \omega_t,\]
where $\omega \in \C_{\R^d}$ and $\alpha_N \in \pb(\C_{\R^d})$ is a given parameter. 
Choosing the empirical measure of the driving noises
\begin{equation} \label{eq:alphaTanaka}
\alpha_N = \pi ( \vec{B}^N ) :=  \frac{1}{N} \sum_{i=1}^N \delta_{B^i} \in \ps ( \C_{\R^d} ), 
\end{equation} 
the particle system \eqref{eq:mckeanvlasov} can then be realized as $X^{i,N}_t := \mathbb{X}^{\alpha_N}_t(B^i )$.
With this construction, the limit $N\to+\infty$ reduces to a mere continuity property of the map $\mathbb{X}^{\alpha_N}$ with respect to its parameter. In the present case, by the law of large numbers, $\alpha_N \to \mathscr{W}$ where $\mathscr{W}$ is the Wiener measure, and it is thus expected that
\[ X^{i,N} \underset{N\to+\infty}{\longrightarrow} \overline{X}^i := \mathbb{X}^\mathscr{W} (B^i ), \quad 1 \leq i \leq N.\]
Since the Brownian motions are i.i.d with common law $\mathscr{W}$, the limit processes $\overline{X}^i$ are independent copies of the nonlinear McKean-Vlasov process $\overline{X}$ defined by
\begin{align*}
\dd\overline{X}_t &= \int_{\R^d} K(\overline{X}_t, y) f_t(\dd y) + B_t,\qquad f_t = \mathrm{Law}(\overline{X}_t).
\end{align*}

The convergence of the empirical measure $\pi ( \vec{X}^N )$ can be also re-casted as a continuity result for the map $\alpha \in \ps ( \C_{\R^d} ) \mapsto \mathbb{X}^\alpha_\# \alpha$ by noticing that $\pi ( \vec{X}^N ) = \mathbb{X}^{\alpha_N}_\# \alpha_N$. Importantly, large deviation results follow almost immediately since, by the contraction principle, the continuity of $\alpha \mapsto \mathbb{X}^{\alpha}_\# \alpha$ induces a large deviation principle for $(\mathrm{Law}(\pi(\vec{X}^N))_{N \geq 1}$ as soon as one is known for $(\mathrm{Law}(\alpha_N))_{N \geq 1}$. Here we recall that given a measurable map $T:E\to F$ between two measurable spaces and a measure $\mu$ on $E$, the pushforward measure denoted by $T_\#\mu$ is a measure on $F$ defined by $T_\#\mu(B) = \mu(T^{-1}(B))$ for any measurable $B\subset F$.

Although Tanaka's approach was historically among the first ones developed in the mean-field literature, it has grown in popularity only recently \cite{backhoff_mean_2020,coghi_pathwise_2020,chaintron_quasicontinuity_2024}. Perhaps one of its main interests is its potential for generalization, as it has been recently shown in \cite{coghi_pathwise_2020}. Indeed, as the core idea is to transform some abstract input data into a particle system, the McKean-Vlasov system appears as just one example that can be naturally generalized by considering more abstract interaction functions, possibly in a pathwise setting, and importantly other types of noise or initial data that may include correlations. We refer to \cite{coghi_pathwise_2020} for a list of models that fall into this generalized framework. The main challenge then becomes to construct a sufficiently regular map $\X$ that can carry the information on the input data to the system of interest. The simplest approach, here and in \cite{coghi_pathwise_2020}, is to define this map as the solution of a fixed point problem in an appropriate functional setting. We notice here that this approach is not restricted to the study of mean-field properties and, strictly speaking, it is much more general than a propagation of chaos result in the sense of Kac. Actually, it is a powerful approach to study parametrized equations in general, not necessarily particle systems. In particular and on that matter, we refer to \cite{chaintron_quasicontinuity_2024} for a discussion on the analogy between particle systems and the Freidlin-Wentzell dynamics using Tanaka's approach. 

The purpose of the present article is to introduce another generalization of Tanaka's idea but, unlike \cite{coghi_pathwise_2020}, our main focus is not to generalize the space or the form of the input data but rather to consider parameters $\alpha_N$ in an abstract space beyond the space of probability measures on path space. One of the main objectives is to embed in this parameter a graph interaction structure among the particles. Particle systems of this kind are often referred to as \emph{non-exchangeable} and their study has recently become an important research topic, owing to the mathematical challenges that this additional structure brings but also to the numerous applications. The prototypical non-exchangeable system, which is a direct extension of the McKean-Vlasov system, is the following
\begin{equation} \label{eq:benchmark}
\d X^{i,N}_t = \frac{1}{N} \sum_{j=1}^N \mathbf{w}^{i,j}_N K ( X^{i,N}_t, X^{j,N}_t ) \d t + \d B^{i}_t, \qquad 1 \leq i \leq N,  
\end{equation}
for a given interaction kernel $K : \R^d \times \R^d \rightarrow \R^d$, 
and where the $( \mathbf{w}_N^{i,j} )_{1 \leq i \leq j \leq N}$ are random weights (typically in $[0,1]$) that keep track of the connections between each pair of particles. These weights are typically interpreted as the components of the adjacency matrix of an underlying graph. Following the terminology of \cite{ayi_largepopulation_2024}, the system \eqref{eq:benchmark} is \emph{non-exchangeable} because the particle $X^{i,N}_t$ interacts with the $X^{j,N}_t$, $1 \leq j \leq N$, through its own \emph{interaction measure} 
\begin{equation} \label{eq:InterMeas}
    \pi^{i,N} := \frac{1}{N} \sum_{j = 1}^N \delta_{(\mathbf{w}^{i,j}_N, X^{j,N})} \in \ps ( [0,1] \times \C_{\R^d} ), 
    \end{equation} 
contrary to the usual setting of mean-field exchangeable systems, where $\pi^{i,N} = \pi ( \vec{X}^N )$ does not depend on the label $i$. Note however that the interaction can still be regarded as of mean-field type because the particle $X^{i,N}_t$ interacts with all the couples $( \mathbf{w}^{i,j}_N, X^{j,N} )$ in the same way through $\pi^{i,N}$.

A simple example for the weights would be $\mathbf{w}^{i,j}_N = W ( i/N, j/N)$, where the map $W : [0,1]^2 \rightarrow [0,1]$ is a (possibly random) graphon. In this case, the (marginal or joint) regularity of $W$ plays a key role to compute the limit of $\pi( \vec{X}^N)$. If $K$ is bounded Lipschitz and $(\xi,\xi') \mapsto W ( \xi, \xi')$ is also jointly Lipschitz-continuous, then~\eqref{eq:benchmark} reduces to the usual framework of mean-field systems with Lipschitz coefficients by considering the augmented particles $Y^{i,N} := (i/N,X^{i,N})$, see for instance \cite{paul_microscopic_2024}. An important challenge is then to leverage the structure of the labeled particle system to avoid the need for regularity with respect to the label input (which should ideally remain an arbitrary choice). In the following, some of our results may require regularity with respect to $\xi$, but we avoid making regularity assumptions with respect to $\xi'$.

A crucial part of the study is to find a proper framework to encode graphs and graph limits. 
In the past few years, several mathematical representations have been introduced. 
As already mentioned, graphons are among the first and most natural ones that have been applied to the mean-field problem in the seminal articles \cite{kaliuzhnyi-verbovetskyi_mean_2018,chiba_mean_2019}. 
Since then, several refinements have been developed, let us for instance mention the important notion of extended graphons \cite{jabin_meanfield_2025} that has been used to treat sparse graphs with only boundedness assumptions -- in particular, no continuity with respect to labels is required.
Graphs can also be represented as operators \cite{gkogkas_graphop_2022} and more generally can be embedded in a measure-theoretic framework with the recent notion of digraph measure \cite{kuehn_vlasov_2022,kuehn_meanfield_2024}. 
Our present work is based on this latter notion that fits well with Tanakas's approach. 
Related to mean-field limits for graphon systems, we also refer to \cite{bayraktar_graphon_2023,crucianelli_interacting_2024,lacker_quantitative_2024,bayraktar_graphon_2025} for further recent results, to \cite{delarue_mean_2017,caines_graphon_2021,carmona_stochastic_2022,lacker_labelstate_2022,cao_probabilistic_2025} in the context of graphon stochastic games and non-exchangeable mean-field control, as well as \cite{bashiri_gradient_2020,bertoli_phase_2025,carrillo_evolution_2025} for long-time asymptotics and gradient flow formulations. 
For more details on applications and comparison of the different mathematical formalisms, as well as many other examples, we refer to the comprehensive review article \cite{ayi_largepopulation_2024} and the references therein. 

Our main contributions are the following. 
\begin{itemize}
    \item We generalize Tanaka's construction and prove well-posedness, mean-field limit and large deviations results under Lipschitz assumptions in an abstract setting.
    
    \item We show that this abstract framework can be specialized to handle network interactions thanks to a careful choice of the parameter space based on the notion of digraph measures. This framework covers both constant and general adaptive (time-evolving) networks, in either deterministic or stochastic settings.
    
    \item For the expository example \eqref{eq:benchmark}, we recover known results for the mean-field limit and prove a new large deviation principle (LDP) for the interaction measure \eqref{eq:InterMeas}. Our assumptions on the graph structure cover the environment noise setting \cite{daipra_mckeanvlasov_1996}, Erd\H{o}s-Rényi graphs similar to \cite{coppini_law_2020,oliveira_interacting_2019} or Lipschitz graphons. Lipschitz-continuity is assumed for the particles but not necessarily for the labels. We further succeeded in obtaining the rate function in relative entropy form as it is the case in known LDPs for the usual empirical measure \cite{daipra_mckeanvlasov_1996,fischer_form_2014}.
    \item We formally characterize the mean-field limit of a class of particle systems with time-evolving weights as a closed system of PDEs. These last results provide a formal answer to some open modeling questions in the literature, but a complete rigorous proof is left for future work. 
\end{itemize}

Since the present framework is concerned with the generality of the graph structure rather than the regularity of coefficients, our results assume Lipschitz-continuity of coefficients with respect to particles -- but not necessarily with respect to the labels --, as it is the easiest way to apply fixed-point results.
Formally, the digraph formalism used in our work could encode almost every type of graph interactions. 
However, our main practical limitation to compute mean-field limits and large deviations is the regularity of the graph, mainly the continuity of the interaction measure $\pi^{i,N}$ given by \eqref{eq:InterMeas} with respect to the particle label $i/N$ -- this excludes non-continuous settings like \cite{jabin_meanfield_2025}.
We will need to assume continuity with respect to the label either for the limit -- as it is case for usual Erd\H{o}s-R\'enyi graphs -- or uniformly at the level of particles using assumptions like \eqref{eq:UnifC} below.
The case of sparse Erd\H{o}s-R\'enyi graphs \cite{delattre_note_2016,oliveira_interacting_2019,medvedev_continuum_2019,coppini_law_2020} is also beyond the present analysis, see Remark~\ref{rem:SparseErdos} below.

Interaction graphs are often used in synchronization models, like the Kuramoto model or for closest-neighbor type interactions -- see e.g. \cite{berglund2007metastability} and the references therein.
However, although they are a natural application case, these models often require a case-by-case analysis as they depend on specific choices of space-rescaling; therefore, precise results for these models are beyond the scope of our generic abstract framework. 

Regarding large deviations, Tanaka's original approach and our present extension of it are restricted to additive noise, contrary to \cite{baldasso_large_2022,gao_large_2024} that cover multiplicative noise for particles on Erd\H{o}s-Rényi graphs.
We refer to \cite{chaintron_quasicontinuity_2024} for a possible way of circumventing this restriction in Tanaka's method.

The present article is organized as follows. In Section~\ref{sec:PartFixed}, we introduce the abstract framework in which particle systems can be realized as the solution of a fixed point equation. We show how it applies to network interactions in Section~\ref{ssec:examples}. The abstract well-posedness, mean-field limit and large deviations results are proved in Section~\ref{sec:Main}. Section~\ref{sec:applications} is devoted to the application of these abstract results to the benchmark example \eqref{eq:benchmark}. Finally we derive a PDE characterization of the mean-field limit of adaptive networks in Section~\ref{sec:pde}. The Appendix~\ref{app:ldp} gathers reminders on large deviation theory and an augmented labeled version of Sanov's theorem is proved in Appendix~\ref{sec:app}.

\section{Particle systems as fixed points} \label{sec:PartFixed}

Our starting point is a fixed-point equation posed in abstract functional spaces (Section~\ref{ssec:abs}). Choosing the appropriate spaces and inputs can lead to various particle systems and in particular to non-exchangeable systems with adaptive weights (Section~\ref{ssec:examples}). The well-posedness of this formulation is postponed to Section~\ref{sec:Main}.

\subsection{Abstract setting} \label{ssec:abs}

Given a vector space $E$ and a time horizon $T >0$, our objective is to construct a parametrized fixed-point map $\mathbb{X}^\alpha : \Omega \to \mathcal{C}_E$ defined by the following equation 
\begin{equation}\label{eq:Xa}\forall \omega\in \Omega, \,\,\forall t\in[0,T],\quad \mathbb{X}^\alpha_t(\omega) = \int_0^t b_s \big( \alpha,\omega,\mathbb{X}^\alpha)\,\dd s + \sigma_t(\omega),
\end{equation}
which depends on the following spaces and objects. 

\begin{itemize}
    \item We require the state space $E$ to be a Banach space $(E, \vert \cdot \vert_E )$, although most examples will stick to the case $E=\R^d$. 
    The space of continuous $E$-valued processes over the fixed time interval $[0,T]$ is denoted by $\mathcal{C}_E$. This space is endowed with the norm $|X |_{\mathcal{C}_E} := \sup_{t\in [0,T]} |X_t|_E$. 
    \item The measurable space $(\Omega,\mathcal{T})$ is the space of inputs. Typical examples are:
    \begin{enumerate}
    \item Driving noise like Brownian motion, in which case $\Omega = \mathcal{C}_E$.
    \item Initial conditions, in which case $\Omega = E$. 
    \item Both at the same time, in which case $\Omega = E \times \mathcal{C}_E$.
    \end{enumerate}
    In the case of non-exchangeable particles, $\Omega$ will also include a set of labels.

    \item The space $\mathcal{A}$ is the space of parameters. We assume that it is a metric space endowed with a distance $d_\mathcal{A}$. We keep it abstract here, as choosing the appropriate space for each example is a key challenge of the present work. 

    \item The measurable map $\sigma : [0,T] \times \Omega \to E$ is used to add stochastic driving terms in the system. 
    Typical examples with the same numbering as above for $\Omega$ are
    \begin{enumerate}
        \item (Brownian motion) $\sigma_t(B) = B_t$.
        \item (Initial condition) $\sigma_t(X_0) = X_0$ for all $t \in [0,T]$.
        \item (Both) $\sigma_t(X_0,B) = X_0 + B_t$.
    \end{enumerate}   

    \item The map $b : [0,T] \times \A \times \Omega \times ( \sigma + \B) \to E$ is a measurable drift term, where $\B = \B (\Omega,\C_E)$ denotes the Banach space of bounded measurable maps $\Omega \rightarrow \C_E$ endowed with the supremum norm and $\sigma+\B$ is the affine space of functions $\omega \mapsto \sigma(\omega) + X(\omega)$ where $X\in\B$. The function $b$ is assumed to be \emph{non-anticipative} in the sense that
    \[b_t (\alpha,\omega, \mathbb{X}) = b_t(\alpha,\omega,\mathbb{X}_{\wedge t}), \]
    where $\mathbb{X}_{\wedge t} := ( \mathbb{X}_{s \wedge t} )_{0 \leq s \leq T}$ is the stopped path at $t$.
\end{itemize}

As a first trivial example, we can consider the single parametrized SDE
\[ \A = \emptyset, \qquad \Omega = [0,1] \times \C_E, \]
together with
\[ b_t (\alpha,\omega,X) = b (X_t (\omega)), \qquad \sigma_t ( \omega ) = \sqrt{\varepsilon} \gamma_t, \]
for a given $\omega = (\varepsilon, \gamma)\in\Omega$ and a drift term $b : E \rightarrow E$.
The fixed-point Equation~\eqref{eq:Xa} builds the flow map of the ODE
\[ X_t = \int_0^t b ( X_s ) \d s + \sqrt{\varepsilon} \gamma_t.\]
When $\gamma$ is the realization of a Brownian motion, this corresponds to the typical dynamics studied by Freidlin and Wentzell. The classical Freidlin-Wentzell theory then focuses on $\varepsilon$ small, whereas in this article we start from a similar fixed-point construction but applied to particle systems with a focus on when the number of particles is large.

\subsection{Particle systems} \label{ssec:examples}

Once the map $\mathbb{X}^\alpha$ is defined in the abstract setting, a particle system can be constructed as follows. Given sequences $(\alpha_N)_{N \geq 1} \subset \mathcal{A}$ and $(\omega^{i,N})_{1 \leq i \leq N} \subset \Omega$, we define the particles
\[ X^{i,N} := \mathbb{X}^{\alpha_N}(\omega^{i,N}), \quad 1 \leq i \leq N, \]
each particle then satisfies
\begin{equation} \label{eq:abstractparticles}
\forall t \in [0,T], \qquad X^{i,N}_t = \int_0^t b_s( \alpha_N, \omega^{i,N}, \mathbb{X}^{\alpha_N} ) \dd s + \sigma_t(\omega^{i,N}).
\end{equation}
In the following, the particle system will be denoted by $\vec{X}^N := ( X^{i,N} )_{1 \leq i \leq N}$.
We shall also write
\[ \pi ( \vec{X}^N ) := \frac{1}{N} \sum_{i=1}^N \delta_{X^{i,N}}, \]
for the empirical measure of the $N$-tuple $\vec{X}^N$. We recall that this empirical measure can be re-written in terms of the inputs, parameters and Tanaka's map as 
\[\pi(\vec{X}^N) = \mathbb{X}^{\alpha_N}_\#\pi(\vec{\omega}^N)\]
where $\pi(\vec{\omega}^N) = \frac{1}{N}\sum_{i=1}^N \delta_{\omega^{i,N}}$. With this construction, each particle is fully characterized by its input data in $\Omega$ while the parameter in $\A$ will be used to define the interactions. As an example to keep in mind, the general McKean-Vlasov particle system 
\[\d X^i_t = \bar{b}_t\big(X^i_t,\pi(\vec{X}^N)\big)\d t + \d B^i_t,\]
with a non-anticipative drift $\bar{b}:[0,T]\times E\times\pb(\mathcal{C}_E)\to E$, independent Brownian motions $B^i\in \mathcal{C}_E$ and initial values $X^i_0\in E$, is a direct extension of Tanaka's seminal construction outlined in the introduction, corresponding here to \eqref{eq:abstractparticles} with 
\[\omega^{i,N}=(X^i_0,B^i)\in \Omega := E\times \mathcal{C}_E,\quad \alpha_N = \frac{1}{N}\sum_{i=1}^N \delta_{\omega^{i,N}}\in \mathcal{A}:=\pb(\Omega),\]
and the drift and diffusion functions
\[b_t(\alpha,\omega,\mathbb{X}) = \bar{b}(\mathbb{X}_t(\omega),\mathbb{X}_\#\alpha),\quad \sigma_t(x,\gamma) = x + \gamma_t.\]
This simple case essentially relies on the fact that the empirical measure of the particle system rules all the interactions and can be simply written as a push-forward of the empirical measure of the given input data. In order to treat non-exchangeable systems of the form \eqref{eq:benchmark}, this strategy needs to be refined to incorporate the two main challenges: each particle has its own interaction measure and this measure depends on a set of weights, given independently and possibly evolving in time. To this end, we propose the following fairly general framework.

In order to encode the individual interaction measures, since one input data in $\Omega$ defines one particle, it is natural to take as the parameter space $\A$ a space of measurable maps on $\Omega$. This map takes its values in the space of probability measures $\pb(\Omega\times\W')$ that enriches the previous space $\pb(\Omega)$ by introducing a new space $\mathcal{W}'$ of \emph{input weights}, typically $\mathcal{W}'=[0,1]$. We can then define the drift term that contains the interactions, here we take it of the form
\begin{equation} \label{eq:GeneralGraph}
\A = \B ( \Omega, \ps ( \W' \times \Omega ) ), \qquad b_t ( \alpha, \omega, \mathbb{X} ) = \overline{b}_t ( \mathbb{X} ( \omega  ), \pi_t ( \alpha, \omega , \mathbb{X} ) ), 
\end{equation} 
where the \emph{interaction measure} associated to an input data $\omega\in\Omega$ is defined by
\[ \pi_t ( \alpha, \omega , \mathbb{X} ) := (W_t (\alpha, \omega, \mathbb{X} ), \mathbb{X} )_\#\alpha(\omega) \in \ps ( \W \times \C_E ).\]
In this definition, we introduce the map $W$ and the \emph{space of weights} $\mathcal{W}$ to model the time evolution of the input weights: for a given triplet $(\alpha,\omega,\mathbb{X})$, the map
\[ W (\alpha, \omega, \mathbb{X}) : \W' \times \Omega \rightarrow \C_\W, \]
takes the input weights, initially in the space $\mathcal{W}'$ and an input data in $\Omega$ and outputs the time-evolution process of the weights that take values in the space $\mathcal{W}$ (possibly different than $\mathcal{W}'$). We use the notational convention 
\[(W_t (\alpha, \omega, \mathbb{X} ), \mathbb{X}) : (\mathbf{w}',\omega')\in\mathcal{W}'\times \Omega \mapsto (W_t (\alpha, \omega, \mathbb{X} )(\mathbf{w}',\omega'),\mathbb{X} (\omega'))\in\mathcal{W}\times\Omega.\]
In this model, the drift term in \eqref{eq:GeneralGraph} is of the form $\bar{b}_t : \C_E \times \pb(\W\times \C_E)\to E$. To ease the presentation, in the following we will mainly focus on binary interactions and scalar weights (see Remark~\ref{rem:general} regarding other models), corresponding to
\begin{equation}\label{eq:binaryscalar}\bar{b}_t:(X,\pi_t)\in\mathcal{C}_E\times {\mathcal{P}(\mathcal{W}\times \mathcal{C}_E)} \mapsto \int_{\mathcal{W}\times \mathcal{C}_E} \mathbf{w} K_t(X,Y) \pi_t(\d \mathbf{w},\d Y),\end{equation}
for a non-anticipative kernel $K:[0,T]\times \mathcal{C}_E\times \mathcal{C}_E \to E$ and where 
\[\mathcal{W}' = [0,1],\quad \mathcal{W}=\R.\]
With this choice, the fixed-point Equation~\eqref{eq:Xa} becomes 
\[\mathbb{X}^\alpha_t(\omega) = \int_0^t\int_{\W'\times\C_E} W_t(\alpha,\omega,\mathbb{X}^\alpha)(\mathbf{w}',\omega') K_t\big(\mathbb{X}^\alpha(\omega),\mathbb{X}^\alpha(\omega')\big)\alpha(\omega)(\d \mathbf{w}',\d\omega') + \sigma_t(\omega).\]

To see that this abstract framework includes the benchmark example, we are left to specify the $\omega^{i,N}\in\Omega$ and the sequence of parameters $\alpha_N\in\A$. We use a digraph measure formalism originally due to \cite{kuehn_vlasov_2022} and adapted to our framework. Here, each particle is determined, not only by its initial state and driving noise but also by its label, which is taken by convention in the unit interval $I=[0,1]$. Thus $\Omega=I\times E\times\C_E$ and we take as inputs
\[\omega^{i,N} = (i/N,X^i_0,B^i)\in\Omega,\quad i\in\{1,\ldots,N\}.\]
For this model, we always consider the driving term to be 
\[\sigma_t(\xi,x,B) := x + B_t.\]
We then introduce the following digraph measure on $\Omega$, for any $\omega = (\xi,x,B)\in\Omega$,
\[\alpha_N(\omega)(\dd \mathbf{w}', \dd\omega') := \frac{1}{N} \sum_{j=1}^N \delta_{( \mathbf{w}^{[\xi]_N,j}_N , \omega^{j,N} ) }(\d \mathbf{w}', \d \omega') \in\pb(\W'\times \Omega), \]
which depends on a $N\times N$ matrix of initial weights $(\mathbf{w}^{i,j}_N)_{i,j}$ and where $[\xi]_N$ denotes the integer $i$ such that $\xi \in [i/N,(i+1)/N)$. This is an extension to our framework of the digraph measure $\eta : \xi \in [0,1] \mapsto \eta^{\xi}(\dd \zeta)\in\mathcal{M}_+([0,1])$ introduced in \cite{kuehn_vlasov_2022} and defined by 
    \[\eta^\xi(\dd\zeta) := \frac{1}{N} \sum_{j=1}^N \mathbf{w}^{[\xi]_N,j}_N \delta_{\frac{j}{N}}(\dd\zeta).\]
At this point, the choice $W_t (\alpha,\omega,\mathbb{X})(\mathbf{w}',\omega')=\mathbf{w}'$ readily leads to \eqref{eq:benchmark}. We refine here this version by adding a time-evolution of the weights. Taking any sufficiently regular function $\phi:E\times E\times \R\to\R$, let us define 
\[W_t(\alpha,\omega,\mathbb{X})(\mathbf{w}',\omega') \equiv W_t(\mathbb{X}(\omega),\mathbb{X}(\omega'),\mathbf{w}') := \mathbf{w}_t,\]
where $\mathbf{w}_t\in\R$ is the solution at time $t$ of the ODE 
\begin{equation}\label{eq:odew}\frac{\d}{\d t}\mathbf{w}_t = \phi(\mathbb{X}_t(\omega),\mathbb{X}_t(\omega'),\mathbf{w}_t), \quad \mathbf{w}_0 = \mathbf{w}'.\end{equation}
Reporting everything in \eqref{eq:abstractparticles}, we then obtain the following particle system that will be the main object of Section~\ref{sec:applications}.
\begin{align}\label{eq:particleweights}
    &\d X^i_t = \frac{1}{N}\sum_{j=1}^N W_t(X^i,X^j,\mathbf{w}^{ij}_N)K_t(X^i,X^j)\d t + \d B^i_t.
\end{align}

\begin{remark}\label{rem:general}
Most of the abstract results presented below could be applied to more general interaction functions. Typically, one could also handle a nonlinear dependence on $\pi$, typically
    \[ \overline{b}_t: (X,\pi) \mapsto \frac{\int_{[0,1] \times \mathcal{C}_E} \mathbf{w} K_t ( X, Y ) \pi (\dd \mathbf{w}, \d Y)}{\int_{[0,1] \times \C_E} \mathbf{w} \pi (\d \mathbf{w}, \d Y)}, \]
    corresponding to the interaction
    \[ b_t(\alpha_N,\omega^{i,N}, \mathbb{X}^{\alpha_N}) = \frac{\sum_{j=1}^N W_t(X^i,X^j,\mathbf{w}^{ij}_N) K_t(X^{i},X^{j})}{\sum_{j=1}^N W_t(X^i,X^j,\mathbf{w}^{ij}_N)},\]
    where the force is normalized on each sub-network analogously to the Motsch-Tadmor model \cite{motsch_new_2011}. Furthermore, the general weight function $W$ or the function $\phi$ in \eqref{eq:odew} may explicitly depend on the variables $\omega'$ and $\alpha$ (not only through the function $\mathbb{X}$). One may consider the latter case to treat more challenging models where the evolution of the weights also depends on mean-field interactions. It would also be easy to add an environmental noise \cite{daipra_mckeanvlasov_1996,coghi_pathwise_2020,baldasso_large_2022} by enriching the input space $\Omega$ with an additional component. 
\end{remark}

\section{Well-posedness and abstract mean-field limit} \label{sec:Main}

The following result constructs the solution $\mathbb{X}^\alpha$ of \eqref{eq:Xa} in the affine space $\sigma + \B$, recalling that $\B = \B (\Omega,\C_E)$ is the Banach space of bounded measurable maps $\Omega \rightarrow \C_E$ endowed with the supremum norm.
The proof can be seen as a variation on \cite[Proposition 6]{coghi_pathwise_2020}.

\begin{theorem}[Well-posedness and regularity]\label{thm:wellposedness}
Assume that the function $b$ is globally bounded in all its arguments.
Given $\alpha\in\A$, if there exists $L_b^\B = L^\B_b (\alpha) >0$ such that for all $(t,\omega) \in [0,T] \times \Omega$,
\begin{equation}\label{eq:LipbX}
\forall (\mathbb{X},\mathbb{Y})\in (\sigma +\B )^2,\quad  \vert b_t (\alpha,\omega,\mathbb{X})- b_t (\alpha,\omega,\mathbb{Y}) \vert \leq L_b^\B \vert \mathbb{X}_{\wedge t} - \mathbb{Y}_{\wedge t} \vert_\B,
\end{equation}
then \eqref{eq:Xa} has a unique solution $\mathbb{X}^\alpha$ in $\sigma + \B$. Moreover, $\mathbb{X}^\alpha$ enjoys the following regularity properties.

\begin{enumerate}
\item \textbf{(Lipschitz in $\Omega$).} 
Fix $\alpha \in \A$ and assume that $(\Omega,d_\Omega)$ is a metric space. If there exist $L^\Omega_b = L^\Omega_b (\alpha) > 0$, $L^\Omega_\sigma > 0$ such that for every $(t,\omega,\omega') \in [0,T] \times \Omega^2$,
\begin{subequations}\label{eq:LipbOmega}
\begin{align}
\vert b_t(\alpha,\omega,\mathbb{X}^\alpha) - b_t (\alpha,\omega',\mathbb{X}^\alpha) \vert_E &\leq L^\Omega_b \big[ \vert \mathbb{X}^\alpha_{\wedge t} ( \omega ) -  \mathbb{X}^\alpha_{\wedge t} ( \omega' ) \vert_{\C_E} + d_\Omega(\omega,\omega') \big], \\
\vert \sigma_{t} ( \omega ) - \sigma_{t} ( \omega' ) \vert_E &\leq L^\Omega_\sigma d_\Omega ( \omega, \omega'),
\end{align}
\end{subequations}
then the map $\mathbb{X}^\alpha :\Omega\to \mathcal{C}_E$ is Lipschitz-continuous with Lipschitz constant 
\[L_\mathbb{X}^\Omega :=(L_\sigma^\Omega + L^\Omega_b T ) \e^{L^\Omega_b T}.\]

\item \textbf{(Lipschitz in $\A$).} Given $\alpha,\beta\in\A$, if $L_\phi^\A = L_\phi^\A (\alpha,\beta) >0$ exists such that
\begin{equation}\label{eq:LipbA} 
\forall \omega \in \Omega, \quad
\vert b_t (\alpha,\omega,\mathbb{X}^\alpha) - b_t (\beta,\omega,\mathbb{X}^\beta ) \vert_E \leq L_b^\mathcal{A} \big[ \vert \mathbb{X}^\alpha_{\wedge t} - \mathbb{X}^\beta_{\wedge t} \vert_{\B} + d_\mathcal{A}(\alpha,\beta) \big],
\end{equation}
then, defining $L_\mathbb{X}^\mathcal{A} := L^\A_b T e^{L^\A_b T}$,
\begin{equation}\label{eq:LipXA}|\mathbb{X}^\alpha - \mathbb{X}^\beta|_\B \leq L_\mathbb{X}^\mathcal{A} d_\A(\alpha,\beta). \end{equation}
In particular, if \eqref{eq:LipbA} holds uniformly for every $\alpha,\beta\in\A$, then $\alpha \mapsto \mathbb{X}^\alpha$ is Lipschitz-continuous.
\end{enumerate}
    
\end{theorem}

\begin{proof}
For $\alpha \in \A$, we define a map $\Phi : \sigma + \B \rightarrow \sigma + \B$ by 
\[ \Phi(\mathbb{X})(t,\omega) := \int_0^t b_s(\alpha,\omega,\mathbb{X}) \dd s + \sigma_t(\omega).\]
From \eqref{eq:LipbX}, we get the bootstrapping relation
\[ \forall t \in [0,T], \quad \vert \Phi_{\wedge t} ( \mathbb{X} ) - \Phi_{\wedge t}  ( \mathbb{Y} ) \vert_\B \leq L^\B_b \int_0^t \vert \mathbb{X}_{\wedge s} -\mathbb{Y}_{\wedge s} \vert_\B \d s, \]
classically implying that an iterate of $\Phi$ is a contraction.
Existence and uniqueness for \eqref{eq:Xa} then stem from the Banach fixed-point theorem.
The regularity properties are then direct consequences of Gronwall estimates:
\begin{enumerate}
\item From \eqref{eq:LipbOmega}, for every $t \in [0,T]$,
\[ \vert \mathbb{X}^\alpha_{\wedge t} (\omega) - \mathbb{X}^\alpha_{\wedge t} (\omega') \vert_{\mathcal{C}_E} \leq L^\Omega_b \int_0^t \big[ |\mathbb{X}^\alpha_{\wedge s} (\omega) - \mathbb{X}^\alpha_{\wedge s}(\omega') |_{\mathcal{C}_E} + d_\Omega(\omega,\omega') \big] \dd s + L_\sigma^\Omega d_\Omega(\omega,\omega'). \]
Gronwall's lemma leads to 
\[ \vert \mathbb{X}^\alpha (\omega) - \mathbb{X}^\alpha (\omega') \vert_{\mathcal{C}_E} \leq (L_\sigma^\Omega + L^\Omega_b T )\e^{L^\Omega_b T}d_\Omega(\omega,\omega'). \]
\item For $\alpha,\beta\in\A$, \eqref{eq:LipbA} entails 
\[ \vert \mathbb{X}^\alpha_{\wedge t} - \mathbb{X}^\beta_{\wedge t} \vert_\B \leq L_b^\A \int_0^t \big[ \vert \mathbb{X}^\alpha_{\wedge s} - \mathbb{X}^\beta_{\wedge s} \vert_\B + d_\mathcal{A}(\alpha,\beta) \big] \d s. \]
Once again, the conclusion follows by Gronwall's lemma. 
\end{enumerate}

\end{proof}

The mean-field limit and the associated large deviation principles are direct corollaries of Theorem~\ref{thm:wellposedness}. As a general principle, the Lipschitz estimates indeed ensure that any convergence result that holds in $\Omega$ and $\mathcal{A}$ will propagate to the particle system. 

\begin{corollary}[Mean-field limit]\label{thm:abstract_meanfield}
Let us consider sequences $( \alpha_N )_{N \ge 1} \subset \A$ and $( \omega^{i,N} )_{1 \leq i \leq N} \subset \Omega$. 
We assume that $\alpha$ and the $\alpha_N$ satisfy \eqref{eq:LipbX}.
Let $X^{i,N} := \mathbb{X}^{\alpha_N}(\omega^{i,N})$, $1 \leq i \leq N$, be the particle system satisfying \eqref{eq:abstractparticles}. 
We eventually assume that \eqref{eq:LipbA} holds for $(\alpha,\beta) = (\alpha,\alpha_N)$ uniformly in $N \geq 1$.

\begin{enumerate} 
\item The Lipschitz estimate \eqref{eq:LipXA} implies
\[ \max_{1 \leq i \leq N} \vert X^{i,N} - \mathbb{X}^{\alpha} ( \omega^{i,N}) \vert_{\C_E} \leq L^\A_{\mathbb{X}} \d_\A ( \alpha_N, \alpha).\]
In particular, if the inputs $\omega^{i,N} \equiv \omega^i$ are independent of $N$ and  $\alpha_N \rightarrow \alpha$ in $\A$, then it implies the pathwise convergence of the particles in $\C_E$
\[X^{i,N} \underset{N\to+\infty}{\longrightarrow} \mathbb{X}^\alpha(\omega^i).\]
\item If $\alpha$ satisfies \eqref{eq:LipbOmega}, then for any $P\in\pb(\Omega)$, it holds that
\[ W_1 ( \pi ( \vec{X}^N ), \mathbb{X}^\alpha_\#P ) \leq L^\A_{\mathbb{X}} d_{\A}(\alpha_N,\alpha) + L^\Omega_{\mathbb{X}} W_1 ( \pi( \vec{\omega}^N ),P).\]
This estimate implies the mean-field limit of the particle system towards $\mathbb{X}^\alpha_\#P$ when $\alpha_N\to \alpha$ in $\mathcal{A}$ and $ \pi( \vec{\omega}^N ) \to P$ in $\ps ( \Omega )$ for the Wasserstein-1 distance. 
\end{enumerate}
\end{corollary}

\begin{proof}
The first point is a direct application of Theorem~\ref{thm:wellposedness}.
For the second one, we decompose
\[ W_1 ( \pi ( \vec{X}^N ), \mathbb{X}^\alpha_\#P ) \leq W_1 ( \pi ( \vec{X}^N ), \mathbb{X}^\alpha_\# \pi( \vec{\omega}^N ) ) + W_1 ( \mathbb{X}^\alpha_\# \pi( \vec{\omega}^N ), \mathbb{X}^\alpha_\#P ). \]
For the first term, we use that $\pi ( \vec{X}^N ) = \mathbb{X}^{\alpha_N}_\# \pi ( \vec{\omega}^N )$ by definition of $( X^{i,N} )_{1 \leq i \leq N}$. 
The bound \eqref{eq:LipXA} then yields
\[  W_1 ( \pi ( \vec{X}^N ), \mathbb{X}^\alpha_\# \pi( \vec{\omega}^N ) ) \leq N^{-1} \sum_{i=1}^N \vert \mathbb{X}^{\alpha_N} ( \omega^{i,N} ) - \mathbb{X}^{\alpha} ( \omega^{i,N} ) \vert_{\C_E} \leq L^\A_{\mathbb{X}} d_{\A}(\alpha_N,\alpha). \]
For the second term, we leverage the stability of $W_1$ under Lipschitz push-forward:
\[ W_1 ( \mathbb{X}^\alpha_\# \pi( \vec{\omega}^N ), \mathbb{X}^\alpha_\#P ) \leq \Vert \mathbb{X}^\alpha \Vert_{\mathrm{Lip}} W_1 ( \pi( \vec{\omega}^N ),P), \]
combined with Theorem~\ref{thm:wellposedness}.
\end{proof}

The first point is analogous to the renowned McKean theorem classically obtained for \eqref{eq:mckeanvlasov} by a coupling method, see e.g. \cite[Theorem 3.1]{chaintron_propagation_2022a}. 
We then leverage the construction $\pi( \vec{X}^N ) = \mathbb{X}^{\alpha_N}_\# \pi( \vec{\omega}^N )$ to prove large deviation results. The following corollary extends large deviation results known for classical McKean-Vlasov systems (see Remark~\ref{rem:ldpclassicalmckean}).

\begin{corollary}[Large deviations] \label{cor:LDP}
Let us assume that:
\begin{enumerate}
    \item Properties \eqref{eq:LipbOmega}-\eqref{eq:LipbA} hold uniformly for every $\alpha, \beta \in \A$.
    \item The $(\alpha_N )_{N \geq 1}$, $( \omega^{i,N} )_{1 \leq i \leq N}$ are random variables such that $( \mathrm{Law}(\alpha_N,\pi( \vec{\omega}^N )))_{N \geq 1}$  satisfies the large deviation principle (LDP) with good rate function $I : \A \times \mathcal{P} ( \Omega ) \rightarrow [0,+\infty]$.
\end{enumerate}
Then, $( \mathrm{Law}(\pi( \vec{X}^N)))_{N \geq 1}$ satisfies the LDP with good rate function
\[ J : \mu \in \mathcal{P} ( \C_E ) \mapsto \inf_{\substack{(\alpha,\nu) \\ \mathbb{X}^\alpha_\# \nu = \mu}} I ( \alpha, \nu). \]   
\end{corollary}

\begin{proof}
The Lipschitz-continuity of $(\alpha,\omega) \mapsto \mathbb{X}^\alpha (\omega)$ given by Theorem~\ref{thm:wellposedness} entails that $(\alpha,\nu) \mapsto \mathbb{X}^\alpha_\# \nu$ is continuous.
The result then follows from the standard contraction principle in large deviation theory, see e.g. Appendix~\ref{app:ldp} and \cite[Theorem 4.2.1]{dembo_large_2009}.
\end{proof}

\begin{remark}[Central limit theorem]
To complete the picture, it is tempting to derive a central limit theorem (CLT) to capture the normal fluctuations of the convergence of $( \pi(\vec{X}^N) )_{N \geq 1}$.
If such a CLT is known for $(\alpha_N,\pi(\vec{\omega}^N))$, then the CLT for $(\pi(\vec{X}^N))_{N \geq 1}$ follows from the delta method if we can perform a first order expansion of $(\alpha,\nu) \mapsto \mathbb{X}^\alpha_\# \nu$.
This would require additional regularity on the drift $b$, which requires a case-by-case study.
In the present article, is is thus irrelevant to write an abstract CLT result, and we rather refer to the existing instances \cite{tanaka_limit_1982,coghi_pathwise_2020,chaintron_quasicontinuity_2024} of this approach in the literature.
\end{remark}

\begin{remark}[Classical McKean-Vlasov system]\label{rem:ldpclassicalmckean}
    As a direct application, we recall the seminal example \eqref{eq:mckeanvlasov} and the results already proved by Tanaka \cite{tanaka_limit_1982}, corresponding in our framework to $\omega=(x,\gamma)\in\Omega := \R^d\times \C_{\R^d}$ and the functions 
    \[ b_t (\alpha,\omega,\X) = \int_\Omega K(\X_t(\omega),\X_t(\omega'))\alpha(\d \omega'),\quad \sigma_t(\omega) = x + \gamma_t. \]
    When the kernel $K:\R^d\times\R^d\to\R^d$ is uniformly bounded and Lipschitz, the only point to check is \eqref{eq:LipbA}. It comes from the decomposition 
    \begin{align*}
        \vert b_t (\alpha,\omega,\mathbb{X}^\alpha) - b_t (\beta,\omega,\mathbb{X}^\beta ) \vert_E \leq \int_\Omega K(\X^\alpha_t(\omega),\X^\alpha_t(\omega'))(\alpha-\beta)(\d \omega') \\
        + \int_\Omega |K(\X^\alpha_t(\omega),\X^\alpha_t(\omega')) - K(\X^\beta_t(\omega),\X^\beta_t(\omega'))|\beta(\d \omega').
    \end{align*}
    By the first point of Theorem~\ref{thm:wellposedness} and the stability of Lispchitz functions by composition, the map $\omega'\mapsto K(\X^\alpha_t(\omega),\X^\alpha_t(\omega'))$ is uniformly Lipschitz and we deduce that the condition \eqref{eq:LipbOmega} is satisfied in $\mathcal{A}=\pb(\Omega)$ endowed with the Wasserstein-1 distance.  Then the first point of Corollary~\ref{thm:abstract_meanfield} gives the classical pathwise coupling estimate 
    \[\E \max_{1 \leq i \leq N} |X^{i,N} - \overline{X}^i|_{\C_{\R^d}} \leq L\E W_1\left(\frac{1}{N}\sum_{k=1}^N \delta_{B^i},\mathscr{W}\right),\]
    where $\overline{X}^{i}$ is the nonlinear McKean-Vlasov process driven by the same initial condition and Brownian motion than $X^{i,N}$. The right-hand side converges to zero by the law of large numbers with a rate computed in \cite{fournier_rate_2015}. The LDP holds similarly by Corollary~\ref{cor:LDP}. Note that $\alpha_N = \pi(\vec{\omega}^N)$ in this case, so we can identify the two variables in Corollary~\ref{cor:LDP}. 
    By Sanov's theorem, $I(\alpha) = H(\alpha|\mathscr{W})$ and we conclude that the particle system satisfies the LDP with good rate function $J$ given above.
    Moreover, the marginal laws $( f_t )_{0 \leq t \leq T}$ of the path law $\X^\mathscr{W}_\#\mathscr{W}$ of the limit McKean-Vlasov process give the weak solution of the nonlinear Fokker-Planck equation 
    \[\partial_t f_t (x) = -\nabla_x\cdot \left(f_t(x)\int_{\R^d}K(x,x')f_t(\d x')\right) + \frac{1}{2}\Delta_x f_t(x).\]
\end{remark}

\section{Applications to particle networks} \label{sec:applications}

In this section we detail how the previous results can be applied to particle systems with network interactions introduced in Section~\ref{ssec:examples}. The main tools are convergence results and large deviation principles for the input digraph parameters that are propagated to the particle systems via the fixed-point map. 

\subsection{Setting}

Let us recall the general setting \eqref{eq:GeneralGraph}
\begin{equation*}
\A = \B ( \Omega, \ps ( \W' \times \Omega ) ), \qquad b_t ( \alpha, \omega, \mathbb{X} ) = \overline{b}_t ( \mathbb{X} ( \omega  ), \pi_t ( \alpha, \omega , \mathbb{X} ) ), 
\end{equation*} 
involving the interaction measure
\[ \pi_t ( \alpha, \omega , \mathbb{X} ) := (W_t (\alpha, \omega, \mathbb{X} ), \mathbb{X})_\#\alpha(\omega) \in \ps ( \W \times \C_E ), \]
for a non-anticipative function $W_t (\alpha, \omega, \mathbb{X} ) : \W' \times \Omega \rightarrow \W$ that is bounded and Lipschitz-continuous uniformly in $(\alpha, \omega, \mathbb{X})$.
We assume that $(\Omega,d_\Omega)$, $(\W,d_\W)$ and $(\W',d_{\W'})$ are metric spaces.

In the rest of this article, for any metric space $(F,d_F)$,
the probability space $\ps ( F )$ is endowed with the bounded-Lipschitz distance defined by
\[ d_{\mathrm{BL}} ( \mu, \nu) := \sup_{\substack{\varphi \in \C ( F) \\
\Vert \varphi \Vert_\infty + \Vert \varphi \Vert_{\mathrm{Lip}} \leq 1}} \int_F \varphi \, \d [ \mu - \nu]. \]

The following lemma reduces the hypotheses on $b$ in Theorem~\ref{thm:wellposedness} to analogous ones on $W$. 

\begin{lemma} \label{lem:VerifInterMeas}
Suppose that the map
\[
\overline{b}_t : \C_E \times \ps ( \W \times \C_E) \rightarrow E, \]
is non-anticipative and Lipschitz-continuous uniformly in $t \in [0,T]$.
For $\alpha \in \A$, if \eqref{eq:LipbX} holds when replacing $b_t$ by $W_t : \A \times \Omega \times (\sigma + \B) \rightarrow \B(\W' \times \Omega, \W )$, then it also holds for $b_t$.
For $\alpha, \beta \in \A$ that are further Lipschitz-continuous, if \eqref{eq:LipbX}-\eqref{eq:LipbOmega}-\eqref{eq:LipbA} hold for $W_t$, then these properties also hold for $b_t$.
\end{lemma}

\begin{proof}
Leveraging the stability of Lipschitz functions under composition, the only delicate point is the behavior of $\pi_t :\A \times \Omega \times (\sigma + \B) \rightarrow \ps ( \W \times \C_E)$.
Under our running assumptions, this behavior follows from the generic decomposition
\begin{align*}
&\int_{\W' \times \Omega} \varphi \, \d  [ \pi_t (\alpha,\omega,\mathbb{X}) - \pi_t (\beta,\tilde\omega,\mathbb{Y}) ] = \\
&\phantom{+} \int_{\W' \times \Omega} \big[ \varphi \big( W_t ( \alpha, \omega, \mathbb{X} ) ( \mathbf{w}' ), \mathbb{X} (\omega' ) \big) - \varphi \big( W_t ( \beta, \tilde\omega, \mathbb{Y} ) ( \mathbf{w}' ), \mathbb{Y} ( \omega') \big) \big] \alpha ( \omega ) (\d \mathbf{w}', \d \omega') \\
&+ \int_{\W' \times \Omega} \varphi \big( W_t ( \beta, \tilde\omega, \mathbb{Y} ) ( \mathbf{w}' ), \mathbb{Y} ( \omega') \big) [ \alpha ( \omega ) -  \beta ( \tilde\omega ) ] (\d \mathbf{w}', \d \omega'), 
\end{align*}
for any bounded-Lipschitz $\varphi \in \C ( \W' \times \Omega )$ and $\mathbb{X}, \mathbb{Y} \in \sigma + \B$.
Using our assumptions on $(\alpha,W_t)$,
Condition \eqref{eq:LipbX} follows by taking $(\beta,\tilde\omega) = (\alpha,\omega)$, yielding existence for $\mathbb{X}^\alpha$ using Theorem~\ref{thm:wellposedness}.
We then specify the above estimate to $(\mathbb{X},\mathbb{Y}) = (\mathbb{X}^\alpha,\mathbb{Y}^\beta)$.
Condition \eqref{eq:LipbOmega} follows by taking $\beta = \alpha$, and Condition \eqref{eq:LipbA} by taking $\tilde\omega=\omega$ and using that $\omega \mapsto \mathbb{X}^\beta (\omega)$ is Lipschitz from the previous condition and Theorem~\ref{thm:wellposedness}.
\end{proof}

For the key example \eqref{eq:binaryscalar} of binary interactions with scalar weights introduced in Section~\ref{ssec:examples}, it is easy to check that the functions $\bar{b}$ and $W$ indeed satisfy the hypotheses of Lemma~\ref{lem:VerifInterMeas} and in particular all the necessary conditions \eqref{eq:LipbX}-\eqref{eq:LipbOmega}-\eqref{eq:LipbA} (instead of $b$). This comes from the fact that $W$ is given by the flow of the ODE~\eqref{eq:odew} with a globally bounded Lipschitz function $\phi$, leading to standard Gronwall estimates. We recall that we always consider a non-anticipative kernel $K_t : \C_E \times \C_E \rightarrow E$ that is globally bounded and Lipschitz uniformly in $t \in [0,T]$. 
Let us further recall that
\begin{align*}\alpha_N(\omega)(\d \mathbf{w}', \dd\omega') &:= \frac{1}{N} \sum_{j=1}^N \delta_{(\mathbf{w}^{[\xi]_N,j}_N,\omega^{j,N})}(\d \mathbf{w}',\d \omega') \in \pb([0,1]\times\Omega), \\
\omega^{i,N} &:= ( i/N,X^i_0,B^i ) \in I \times \R^d\times \mathcal{C}_{\R^d} =: \Omega,
\end{align*}
with the convention $\omega = (\xi, x, \gamma)$, $(X^i_0,B^i)_{i \geq 1}$ being an i.i.d. sequence with law $f_0(\dd x) \otimes\mathscr{W}(\dd \gamma)$, and $(\mathbf{w}^{i,j}_{N})_{1 \leq i,j \leq N}$ being deterministic for simplicity -- although what follows would still work for random weights that are independent of $(X^i_0,B^i)_{i \geq 1}$. Note that $\alpha_N(\omega)\equiv\alpha_N(\xi)$ only depends on the variable $\xi$.
Thanks to Lemma~\ref{lem:VerifInterMeas} and Theorem~\ref{thm:wellposedness}, the map $\X^{\alpha_N}$ is well-defined, which makes sense of the particle system \eqref{eq:particleweights} defined by $X^{i,N} := \mathbb{X}^{\alpha_N}(\omega^{i,N})$. 

In order to prove a mean-field limit and a LDP, we are left to study the asymptotic behavior of the sequences $(\alpha_N)_N\subset \mathcal{A}$ and $(\pi(\vec{\omega}^N))_N\subset\pb(\Omega)$.
To do so, we make the following crucial uniform continuity assumption on the sequence of digraph measures:
\begin{equation} \label{eq:UnifC}
\sup_{N \geq 1} \sup_{\vert \xi - \xi ' \vert \leq \eta} d_{\mathrm{BL}} ( \alpha_N (\xi ), \alpha_N (\xi ') ) \xrightarrow[\eta \rightarrow 0]{\text{a.s.}} 0. 
\end{equation}
This property holds in particular if $\mathbf{w}^{i,j}_N = W ( i/N,j/N )$ is constructed from a deterministic graphon $W : I^2 \rightarrow [0,1]$, with $\xi \mapsto W(\xi,\xi')$ Lipschitz uniformly in $\xi'$. Note that, as discussed in the introduction, we do \emph{not} need to assume that $W$ is jointly Lipschitz in both variables $(\xi,\xi')$.

The main tool is the following lemma.

\begin{lemma} \label{lem:CVDig}
Let us assume that the initial graph sequence converges in the sense of digraph measures in $\ps ( [0,1] \times I)$,
\begin{equation} \label{eq:CVdigr}
\eta^\xi_N := \frac{1}{N} \sum_{j=1}^N \delta_{( \mathbf{w}^{[\xi]_N,j}_N, j/N )} \xrightarrow[N\to+\infty]{} \eta^\xi,
\end{equation} 
where $[\xi]_N$ denotes the integer $i$ such that $\xi\in[i/N,(i+1)/N)$ and the convergence is understood for the $\infty$-BL distance 
\[d_{\infty, \mathrm{BL}} (\eta_N,\eta) := \sup_{\xi \in [0,1]} d_{\mathrm{BL}}(\eta_N^\xi,\eta^\xi) = \sup_{\xi \in [0,1]} \sup_{\substack{\varphi \in \C ( [0,1] \times [0,1]) \\
\Vert \varphi \Vert_\infty + \Vert \varphi \Vert_{\mathrm{Lip}} \leq 1}} \langle \eta^\xi_N,\varphi\rangle - \langle \eta^\xi,\varphi\rangle.\]
Then, under \eqref{eq:UnifC}-\eqref{eq:CVdigr}, $(\alpha_N)_{N \geq 1}$ a.s. converges in $d_{\infty,\mathrm{BL}}$ towards 
\begin{equation} \label{eq:LimDig}
\alpha(\omega)( \d \mathbf{w}', \d \xi', \d x ', \d \gamma') = \eta^\xi(\d \mathbf{w}', \d \xi') \otimes f_0(\d x') \otimes \mathscr{W}(\d \gamma'), 
\end{equation} 
where $\omega = (\xi,x,\gamma)$ -- note however that $\alpha(\omega) \equiv \alpha(\xi)$ does not depend on $(x,\gamma)$. Here we recall that $f_0$ is the initial distribution of the particles. 
\end{lemma}

\begin{proof}
Set $Y^{\xi,j}_N := (\mathbf{w}_N^{[\xi]_N,j}, j/N)$, $Z^j := (X^j_0,B^j)$, so that
\[ \eta^\xi_N = \frac{1}{N} \sum_{j=1}^N \delta_{Y^{\xi,j}_N}, \qquad \alpha_N ( \omega ) = \frac{1}{N} \sum_{j=1}^N \delta_{(Y^{\xi,j}_N,Z^j)}. \]
The $( Y^{\xi,j}_N )_{1 \leq j \leq N}$ are deterministic and the $(Z^j)_{j \geq 1}$ are i.i.d.
For a fixed bounded Lipschitz $\varphi : [0,1] \times \Omega \rightarrow \R$ and $\xi \in [0,1]$, $( \varphi ( Y^{\xi,j}_N, Z^{j} ) - \E [ \varphi ( Y^{\xi,j}_N, Z^{j} )])_{1 \leq j \leq N}$ is a triangular array of row-wise independent centered random variables that are essentially bounded by $2 \lVert \varphi \rVert_\infty$.
From \cite[Theorem 2.1,Corollary (2.12)]{hu_strong_1997}, we deduce that 
\begin{equation} \label{eq:InterAverageCV}
\frac{1}{N} \sum_{j=1}^N \big[ \varphi ( Y^{\xi,j}_N, Z^j ) - \E [ \varphi ( Y^{\xi,j}_N, Z^1 ) ] \big] \xrightarrow[N \rightarrow +\infty]{\text{a.s.}} 0, 
\end{equation} 
using that the $( Z_j )_{j \geq 1}$ are i.i.d. 
Since $y \mapsto \E [ \varphi ( y, Z^1 ) ]$ is bounded Lipschitz, \eqref{eq:CVdigr} implies that
\begin{equation} \label{eq:InterWeakCV}
\frac{1}{N} \sum_{j=1}^N \E [ \varphi ( Y^{\xi,j}_N, Z^1 ) ] \xrightarrow[N \rightarrow +\infty]{} \int_{[0,1] \times \Omega} \varphi \, \d \alpha ( \omega ). 
\end{equation} 
For any fixed $\omega = (\xi, w, \gamma)$, this shows that almost surely $( \alpha_N (\omega))_{N \geq 1}$ weakly converges towards $\alpha 
(\omega)$.
Defining
\[ \overline{P}_N := \frac{1}{N} \sum_{j=1}^N \mathrm{Law} (Y^{\xi,j}_N,Z^j), \]
we notice that \eqref{eq:InterWeakCV} shows convergence for $( \overline{P}_N )_{N \geq 1}$, and \emph{a fortiori} tightness.

To make the convergence of $\int \varphi \, \d \alpha_N (\omega)$ uniform in $\varphi$, we now adapt \cite[Proof of Theorem 1]{wellner_glivenkocantelli_1981}. For $\varepsilon > 0$, tightness for $( \overline{P}_N )_{N \geq 1}$ gives a compact set $K \subset [0,1] \times \Omega$ such that $\overline{P}_N ( K ) \geq 1 - \varepsilon$, for every $N \geq 1$. Let $K^\varepsilon := \{ x \in [0,1] \times \Omega, \; \d_{[0,1] \times \Omega} (x,K) \leq \varepsilon \}$. The set
$\{ \varphi \in \C (K ), \, \lVert \varphi \rVert_\infty + \lVert \varphi \rVert_{\mathrm{Lip}} \leq 1 \big\}$
being compact, it contains a finite $\varepsilon$-covering $(\varphi_i)_{1 \leq i \leq m}$ such that for every $\varphi \in \C ([0,1] \times \Omega)$ with $\lVert \varphi \rVert_\infty + \lVert \varphi \rVert_{\mathrm{Lip}} \leq 1$, there is a $i$ such that
\[ \sup_{x \in  K} \vert \varphi (x) - \varphi_i (x) \vert \leq \varepsilon, \qquad \sup_{x \in  K^\varepsilon} \vert \varphi (x) - \varphi_i (x) \vert \leq 3 \varepsilon. \]
Setting $g (x) := \max \{ 0, 1 - \varepsilon^{-1} d_{[0,1] \times \Omega} (x,K) \}$, $g$ is bounded Lipschitz and $\1_{K} \leq g \leq \1_{K^\varepsilon}$, so that
\[ \forall N \geq 1, \qquad \overline{P}_N (K^\varepsilon) \geq \int_{K^\varepsilon} g \geq \overline{P}_N ( K ) \geq 1 - \varepsilon. \]
Using \eqref{eq:InterAverageCV} with $\varphi = g$, we get a random $N_\varepsilon \geq 1$ such that almost surely
\begin{align*}
\forall N \geq N_\varepsilon, \quad \alpha_N (\omega)(K^\varepsilon) \geq \int_{K^\varepsilon} g \, \d \alpha_N (\omega) &= \int_{K^\varepsilon} g \, \d [ \alpha_N (\omega) - \overline{P}_N] + \int_{K^\varepsilon} g \, \d \overline{P}_N (\omega) \\
&\geq - \varepsilon + (1-\varepsilon) = 1 - 2 \varepsilon.
\end{align*} 
The computation \cite[Top of page 311]{wellner_glivenkocantelli_1981} then yields without any change that almost surely
\[ \forall N \geq N_\varepsilon, \quad \sup_{\substack{\varphi \in \C ( [0,1] \times [0,1]) \\
\Vert \varphi \Vert_\infty + \Vert \varphi \Vert_{\mathrm{Lip}} \leq 1}} \langle \alpha_N (\omega),\varphi\rangle - \langle \overline{P}_N,\varphi\rangle \leq 10 \varepsilon. \]
This shows that $d_{\mathrm{BL}}(\alpha_N (\omega),\overline{P}_N) \rightarrow 0$ almost surely.
Since \eqref{eq:CVdigr} makes the convergence \eqref{eq:InterWeakCV} uniform in $\varphi$ with $\Vert \varphi \Vert_\infty + \Vert \varphi \Vert_{\mathrm{Lip}} \leq 1$, we deduce the almost sure convergence $d_{\mathrm{BL}}(\alpha_N (\omega),\alpha(\omega)) \rightarrow 0$ for every $\omega = (\xi,\mathbf{w},\gamma) \in \Omega$.
Since $[0,1]$ is compact and $\alpha_N (\omega)$ only depends on $\omega$ through $\xi \in [0,1]$, the uniform continuity \eqref{eq:UnifC} makes this convergence uniform in $\omega$ as desired.
\end{proof}

Note that the proof works the same if the weights are random and the convergence assumption \eqref{eq:CVdigr} holds almost surely.  

\subsection{Mean-field limit and Large Deviations}

The following result is a direct application of Corollary~\ref{thm:abstract_meanfield}.

\begin{proposition}[Mean-field limit]\label{pro:MFdig}
If $d_{\infty, \mathrm{BL}}(\alpha_N,\alpha) \rightarrow 0$ a.s., then
\[ \max_{1 \leq i \leq N} \vert X^{i,N} - \mathbb{X}^{\alpha} ( \omega^{i,N}) \vert_{\C_{\R^d}} \xrightarrow[N \rightarrow + \infty]{\text{a.s.}} 0  \]

If moreover $\omega \mapsto \alpha ( \omega )$ is Lipschitz, then
\[ W_1 ( \pi ( \vec{X}^N ), \mathbb{X}^\alpha_\#(\mathrm{Leb_I} \otimes f_0 \otimes \mathscr{W})) \xrightarrow[N \rightarrow + \infty]{\text{a.s.}} 0, \]
where $\mathrm{Leb}_I$ denotes the Lebesgue measure on $I=[0,1]$. 
\end{proposition}

A sufficient condition for the a.s. convergence $d_{\infty, \mathrm{BL}}(\alpha_N,\alpha) \rightarrow 0$ is given in Lemma~\ref{lem:CVDig} and includes the case where $(\mathbf{w}^{i,j}_N )_{N \geq 1}$ is given by a sequence of uniformly Lipschitz graphons. Note that $\xi \mapsto \eta^\xi$ being Lipschitz implies the same property for $\omega \mapsto \alpha (\omega)$ given by \eqref{eq:LimDig}. Another sufficient condition \eqref{eq:EXpCV} is given by Theorem~\ref{thm:LabSanov} in the Appendix. It includes the Erd\H{o}s-Rényi graph where $(\mathbf{w}^{i,j}_N )_{N \geq 1}$ are i.i.d. Bernoulli variables independent of $N$. Indeed, with the notations of Theorem~\ref{thm:LabSanov}, this case corresponds to $\mathcal{U} := \{0,1 \}$, $Y^{i,N} = (X^i_0,B^i)\in E\times\C_E:=F$ and the convergence condition \eqref{eq:EXpCV} reduces to the Riemann sum convergence theorem with, in the limit $W(\xi,\xi',u) = u$ and $\mu^{\xi'}(\d u) =  (\frac12\delta_0+\frac12\delta_1)(\d u)$.

For any $\xi \in I$, the next result computes the large deviations of the interaction measure
\[ \pi_N ( \xi ) := \frac{1}{N} \sum_{j=1}^N \delta_{( \mathbf{w}_N^{[\xi]_N,j}, X^{j,N})} \in \ps ( [0,1] \times \C_{\R^d} ). \]
The large deviations of $\pi(\vec{X}^N)$ follow by the contraction principle.
Note that Theorem~\ref{thm:LabSanov} would give a sufficient condition to apply Corollary~\ref{cor:LDP}, but the following theorem is stronger and provides a better expression of the rate function.
Before stating the theorem, we first introduce an auxiliary map that will be useful in the following.
Applying Theorem~\ref{thm:wellposedness} with $\A = \ps ( [0,1] \times \C_{\R^d})$, for any $\pi \in \ps ( [0,1] \times \C_{\R^d})$ we get a map $\mathbb{Y}^\pi \in \B ( \Omega, \C_{\R^d} )$ such that
\begin{equation} \label{eq:AuxFixedLD}
\forall \omega\in \Omega, \,\,\forall t\in[0,T],\quad \mathbb{Y}^\pi_t(\omega) = \int_0^t \int_{[0,1] \times \C_{\R^d}} \mathbf{w} K_s ( \mathbb{Y}^\pi ( \omega), Y )\, \pi ( \d \mathbf{w}, \d Y)  \dd s + \sigma_t(\omega).
\end{equation}
Leveraging \eqref{eq:binaryscalar} and the uniqueness in Theorem~\ref{thm:wellposedness}, we notice that
\begin{equation} \label{eq:ContractTrick}
\forall \omega =( \xi, x, \gamma ) \in \Omega, \qquad \mathbb{X}^\alpha ( \omega ) = \mathbb{Y}^{(\mathrm{Id},\mathbb{X}^\alpha)_\# \alpha(\omega)} (\omega) = \mathbb{Y}^{(\mathrm{Id},\mathbb{X}^\alpha)_\# \alpha(\xi,0,0)} (\omega). 
\end{equation} 
To the best of our knowledge, the following result is new in the literature. 

\begin{theorem} \label{thm:DigLDP}
Let us denote $F = \R^d \times \C_{\R^d}$ and $Y^{j,N} := (X^{j}_0,B^j)\in F$ and assume that there exists a deterministic $L > 0$ such that the following conditions hold.
\begin{enumerate}
    \item For every $N \geq 1$, $\omega \mapsto \alpha_N (\omega)$ is $L$-Lipschitz.
    \item  For some measurable space $\mathcal{U}$, there exist $\mu \in \B ([0,1], \ps ( \mathcal{U} \times F))$ and $W \in \B ( I^2 \times \mathcal{U}, [0,1] )$ such that for every $\xi'\in I$, $u\in\mathcal{U}$ the map $\xi \mapsto W(\xi,\xi',u)$ is $L$-Lipschitz and for every $\phi \in L^1 (I,\mathrm{BL}([0,1]\times I \times F))$, 
    \begin{multline*} 
    \frac{1}{N} \sum_{j=1}^N \log \E \exp \bigg[ \int_I \phi ( \xi, \mathbf{w}^{[\xi]_N,j}_N , j/N, Y^{j,N} ) \d \xi \bigg] \xrightarrow[N \rightarrow + \infty]{} \\
    \int_I \log \bigg\{ \int_{{\mathcal{U} \times F}} \exp \bigg[ \int_I \phi ( \xi, W ( \xi, \xi', u ) , \xi', y ) \d \xi \bigg] \mu^{\xi'}(\d u,  \d y) \bigg\}\d \xi'.
    \end{multline*} 
\end{enumerate}
Then for every $\xi \in I$, $(\mathrm{Law}(\pi_N(\xi)))_{N \geq 1}$ satisfies the LDP in $(\ps ( [0,1] \times \C_{\R^d}), d_{\mathrm{BL}} )$ with good rate function
\[ K : \pi \mapsto H ( \pi \vert (\mathrm{Id},\mathbb{Y}^\pi)_\#  \overline{\alpha}(\xi,0,0) ), \]
where the a.s. limit of $( \alpha_N )_{N \geq 1}$ is $\overline{\alpha} : 
(\xi,x,\gamma) \mapsto \int_{\mathcal{U}} \delta_{W(\xi,\xi',u)} ( \d \mathbf{w})  \d \xi' \mu^{\xi'}(\d u, \d x',\d \gamma')$.
\end{theorem}

Note that the condition in the second point corresponds to the condition \eqref{eq:EXpCV} of the labeled Sanov Theorem~\ref{thm:LabSanov}. 

We refer to the Appendix~\ref{app:ldp} for the definition \eqref{eq:defH} of the relative entropy functional~$H$ and for other useful results on large deviations. 
Beyond the Erd\H{o}s-R\'enyi graph previously mentioned, a possible application case is $\mathbf{w}_N^{i,j} = W(i/N,j/N,U^{j})$ for an i.i.d. sequence of random variables $(U^j)_{j \geq 1}$ and a bounded $W$ that is continuous with respect to its first two arguments and uniformly Lipschitz with respect to its first one.
Indeed, the convergence \eqref{eq:EXpCV} then reduces to the Riemann sum convergence theorem. In particular, this covers the random environment model of \cite{daipra_mckeanvlasov_1996}.

\begin{proof}
Let $\A_L \subset \B ( \Omega, \ps ( [0,1] \times \Omega ) )$ denote the subset of bounded $L$-Lipschitz functions.
We are going to work in $\A_L$ rather than $\B ( \Omega, \ps ( [0,1] \times \Omega ) )$, so that the needed conditions \eqref{eq:LipbOmega}-\eqref{eq:LipbA} are satisfied uniformly for every $\alpha, \beta \in \A_L$, using Lemma~\ref{lem:VerifInterMeas}.

Using the trivial continuous injection $L^\infty ( I, \ps ( [0,1] \times \Omega ) ) \rightarrow L^\infty ( \Omega, \ps ( [0,1] \times \Omega ) )$, Theorem~\ref{thm:LabSanov} and the contraction principle give that $( \mathrm{Law} ( \alpha_N ) )_{N \geq 1}$ satisfies the LDP in $L^\infty ( \Omega, \ps ( [0,1] \times \Omega )$ with good rate function
\[
\overline{I} ( \alpha) =  
\begin{cases}
I( \alpha ) \qquad &\text{if  } \alpha ( \xi, x, \gamma ) = \alpha ( \xi,0,0), \\
+ \infty \qquad &\text{otherwise},
\end{cases}
\]
see Theorem~\ref{thm:LabSanov} for the definition of $I$.
Here $\alpha$ only depends on the $\xi$ variable so $\alpha ( \xi,0,0)$ is taken as an arbitrary representative. 

Since $\xi \mapsto W ( \xi, \xi', u)$ is $L$-Lipschitz, the expression of $I$ given by Theorem~\ref{thm:LabSanov} implies that the rate function $\overline{I}$ is only finite on those $\alpha \in L^\infty ( \Omega, \ps ( [0,1] \times \Omega ))$ that are $L$-Lipschitz, so that $( \mathrm{Law} ( \alpha_N ) )_{N \geq 1}$ actually satisfies the LDP in $(\A_L, d_{\infty,\mathrm{BL}})$ with good rate function $\overline{I}$. 

For any fixed $\xi \in I$, we then use the contraction principle for the continuous map 
\[ \alpha \mapsto ( \mathrm{Id}, \mathbb{X}^\alpha)_\# \alpha(\xi,0,0), \]
recalling that $(\alpha,\omega) \mapsto \mathbb{X}^\alpha (\omega)$ is Lipschitz-continuous from Theorem~\ref{thm:wellposedness}.
This transformation turns $\alpha_N$ into $\pi_N ( \xi )$, so that $(\mathrm{Law}(\pi_N ( \xi )))_{N \geq 1}$ satisfies the LDP in $\ps ( [0,1] \times \C_{\R^d} )$ with good rate function
\[ K : \pi \mapsto \inf H ( \nu \vert \d \xi' \mu^{\xi'} ), \]
where we minimize over $(\alpha,\nu)$ satisfying
\[ ( \mathrm{Id},\mathbb{X}^\alpha)_\# \alpha ( \xi,0,0 ) = \pi, \qquad \alpha(\cdot)( \d \mathbf{w}, \d \omega') = \int_{\mathcal{U}} \delta_{W(\cdot,\xi',u)} ( \d \mathbf{w}) \nu ( \d u, \d \omega'), \]
with $\omega' = (\xi',x',\gamma')$ and $K(\pi) = +\infty$ if no such $(\alpha,\nu)$ exists.
Using \eqref{eq:ContractTrick},
we notice that
\[ K ( \pi ) = \inf \; H ( \rho \vert \delta_{W(\xi,\xi',u)} ( \d \mathbf{w}) \d \xi' \mu^{\xi'}(\d u, \d x', \d \gamma') ), \]
where we now minimize over
$( \alpha(\xi,0,0),\rho)$ satisfying
\[ ( \mathrm{Id},\mathbb{Y}^\pi)_\# \alpha ( \xi,0,0 ) = \pi, \quad \alpha(\xi,0,0)( \d \mathbf{w}, \d \xi', \d x', \d \gamma') = \int_{\mathcal{U}} \rho ( \d u, \d \mathbf{w}, \d \xi',\d x', \d \gamma'). \]
We emphasize that we only need to minimize over the measure $\alpha (\xi,0,0) \in \ps ([0,1] \times \Omega)$, and not anymore over the entire function $\omega \mapsto \alpha ( \omega )$.
From the marginal disintegration (Appendix~\ref{app:ldp} and \cite[Theorem D.13]{dembo_large_2009}), we have
\begin{multline*}
H ( \rho \vert \delta_{W(\xi,\xi',u)} ( \d \mathbf{w}) \d \xi' \mu^{\xi'}(\d u, \d x') ) = H ( \alpha ( \xi,0,0 ) \vert \overline\alpha ( \xi,0,0 ) ) \\ + \int_{{[0,1]\times }\Omega} H ( \rho^{{w,}\omega'} \vert \overline\mu^{{w,}\omega'} ) \alpha ( \xi,0,0) ({\d \mathbf{w},}\d \omega'), 
\end{multline*}
using the disintegrations 
\[\delta_{W(\xi,\xi',u)} ( \d \mathbf{w}) \d \xi' \mu^{\xi'}(\d u, \d x') = \overline\mu^{\mathbf{w},\omega'} ( \d u) \overline\alpha ( \xi, 0,0 ) ( \d \mathbf{w}, \d \omega' ),\] and 
\[\rho ( \d u, \d \mathbf{w}, \d \omega' ) = \rho^{\mathbf{w},\omega'} ( \d u) \alpha ( \xi,0,0 ) ( \d \mathbf{w}, \d \omega' ).\]
Since the second term on the right-hand side is always nonnegative and can be canceled independently by minimizing over $(\mathbf{w},\omega') \mapsto \rho^{\mathbf{w},\omega'}$, we get
\[ K ( \pi ) = \inf_{( \mathrm{Id},\mathbb{Y}^\pi)_\# \alpha ( \xi,0,0 ) = \pi} H ( {\alpha}(\xi,0,0) \vert \overline{\alpha}(\xi,0,0) ), \]
where we only minimize over $\alpha (\xi,0,0) \in \ps ([0,1] \times \Omega)$. The desired expression for $K$ then results from the contraction property of relative entropy proved in \cite[Lemma A.1]{fischer_form_2014} and recalled in Appendix~\ref{app:ldp}.
\end{proof}

\begin{remark}
    As noted in the proof, the second condition of the theorem implies the LDP for $(\mathrm{Law}(\alpha_N))_{N\geq 1}$ hence the convergence of the sequence
    \[\alpha_N \underset{N\to+\infty}{\longrightarrow} \overline{\alpha},\]
    almost surely for the $d_{\infty,\mathrm{BL}}$ topology in $\mathcal{B}(I,\pb([0,1]\times \Omega))$ with, for any $\xi\in I$,
    \[\overline{\alpha}(\xi)(\d \mathbf{w}, \d\omega') = \int_{\mathcal{U}} \delta_{W(\xi,\xi',u)}(\d \mathbf{w})\d\xi'\mu^{\xi'}(\d u , \d x',\d\gamma'),\quad \omega'=(\xi',x',\gamma')\in\Omega.\]
    Proposition~\ref{pro:MFdig} thus implies the convergence of $\pi(\vec{X}^N)$ towards $\X^{\overline{\alpha}}_\#(\mathrm{Leb}_I\otimes f_0\otimes\mathscr{W})$. However, the LDP also implies the convergence of the sequence $(\pi_N(\xi))_{N\geq 1}$ towards the unique zero $\overline{\pi}^\xi$ of $K$ which satisfies the relation 
    \[\overline{\pi}^\xi = (\mathrm{Id},\mathbb{Y}^{\overline{\pi}^\xi})_\#\overline{\alpha}(\xi).\]
    Since $\pi(\vec{X}^N)$ is the second marginal of $\pi_N(\xi)$ this implies that 
    \[\pi(\vec{X}^N) \underset{N\to + \infty}{\longrightarrow} \hat{\pi}(\d Y) := \int_{\mathbf{w} \in[0,1]}\overline{\pi}^\xi(\d \mathbf{w}, \d Y),\]
    where $\hat{\pi}$ is independent of $\xi$ as shown below. As a sanity check, we thus need to verify the consistency of the two limits, namely that
    \[\hat{\pi} = \X^{\overline{\alpha}}_\#(\mathrm{Leb}_I\otimes f_0\otimes\mathscr{W}).\]

    First we note that the almost sure convergence of the sequence $(\alpha_N)_{N\geq1}$ also implies the almost sure convergence of the digraph sequence
    \[\eta_N(\xi)(\d \mathbf{w},\d \xi') = \frac{1}{N}\sum_{j=1}^N \delta_{(\mathbf{w}^{[\xi]_N,j}_N,\frac{j}{N})}= \int_{x'\in \R^d, \gamma'\in\C_{\R^d}} \alpha_N(\xi)(\d \mathbf{w}, \d\xi',\d x',\d\gamma')\to\eta^\xi,\]
    where
    \[\eta^\xi(\d \mathbf{w},\d\xi') := \int_{\mathcal{U}}\delta_{W(\xi,\xi',u)}(\d \mathbf{w})\left(\int_{\R^d\times \C_{\R^d}}\mu^{\xi'}(\d u, \d x',\d\gamma')\right)\d\xi'.\]
    Using Lemma~\ref{lem:CVDig} (with the almost sure convergence of $\eta_N^\xi$ instead of a deterministic convergence), we deduce that
    \[\overline{\alpha}(\xi)(\d \mathbf{w}, \d \omega') = \eta^\xi(\d \mathbf{w}, \d \xi')\otimes f_0(\d x')\otimes \mathscr{W}(\d\gamma').\]

    Secondly, by definition \eqref{eq:AuxFixedLD} of $\mathbb{Y}^{\overline{\pi}^\xi}$ and the uniqueness in Theorem~\ref{thm:wellposedness}, it holds that $\X^{\overline{\alpha}} (\omega ) = \mathbb{Y}^{\overline{\pi}^\xi} (\omega)$. 
    
    Thus, for any test function $\Phi$ on $\C_{\R^d}$, 
    \begin{align*}
        \int_{\C_{\R^d}}\Phi(Y) \hat{\pi}(\d Y) &= \int_{\C_{\R^d}}\Phi(\X^{\overline{\alpha}}(\omega'))\int_{\mathbf{w}\in[0,1]}\overline{\alpha}(\d \mathbf{w}, \d\omega')\\
        &=\int_{\C_{\R^d}}\Phi(\X^{\overline{\alpha}}(\omega'))\int_{\mathbf{w}\in[0,1]}\eta^\xi(\d \mathbf{w},\d\xi')\otimes f_0(\d x')\otimes \mathscr{W}(\d\gamma'),
    \end{align*}
and we can easily check from the above definition of $\eta^\xi$ that indeed, for any $\xi\in I$,
\[\int_{\mathbf{w}\in[0,1]}\eta^\xi(\d \mathbf{w},\d\xi') = \d\xi',\]
which yields the desired result. 
\end{remark}

\begin{remark}[Sparse Erd\H{o}s-R\'enyi] \label{rem:SparseErdos}
Our framework, and in particular the condition in Theorem~\ref{thm:DigLDP}, covers all the Erd\H{o}s-R\'enyi models where the weights are Bernoulli variables with a fixed expectation independent of $N$ (or more generally which converges towards a given value). It does not work however for the so-called sparse case where this expectation vanishes $p_N\to0$ but the interaction is at the same time rescaled by $p_N^{-1}$ to remain of order one in expectation \cite{delattre_note_2016,medvedev_continuum_2019,oliveira_interacting_2019,coppini_law_2020}. Since that kind of graph is typically constant (i.e. non adaptive in time, unlike the cases that we consider), such model could be written more concisely in Tanaka's abstract framework using the same input space $\Omega$ but the parameter space $\A = \B(\Omega,\M(\Omega))$, the parameter sequence
\[\alpha_N(\omega)(\dd\omega') = \frac{1}{N}\sum_{j=1}^N \frac{\mathbf{w}^{[\xi]_N,j}_N}{p_N}\delta_{\omega^{j,N}}\in\M(\Omega),\]
and the fixed-point equation
\[\X^\alpha(\omega) = \int_0^t\int_{\C_E} K_t(\X^\alpha(\omega),\X^\alpha(\omega'))\alpha(\omega)(\d\omega') + \sigma_t(\omega).\]
Proving the mean-field limit and the LDP then boils down to the study of the asymptotic behavior of $(\alpha_N)_N$. Although it seems plausible to prove that $\alpha_N$ converges in law towards the empirical measure of the input data (and hence the mean-field limit is the classical McKean-Vlasov process without network interactions), a complete proof and the LDP do not seem easy to obtain in the $\infty$-BL topology. 
\end{remark}

\section{Characterization of the limit}\label{sec:pde}

It remains to characterize the limiting distribution $\mathbb{X}^\alpha_\#(\mathrm{Leb}_I\otimes f_0\otimes\mathscr{W})$ in Proposition~\ref{pro:MFdig}, either as the law of the solution of a stochastic differential equation or as the solution of a partial differential equation problem. The latter provides a convenient and closed characterization but the existence of such PDE problem seems to heavily depend on the particular form of the interactions. We will provide the most general models that we have been able to derive. The starting point is the SDE satisfied by $X^\xi_t := \mathbb{X}^\alpha_t(\omega)$ where $\xi\in I $ is a fixed label, $\omega = (\xi,X_0,B)$ and $(X_0,B)\sim f_0\otimes \mathscr{W}$. Then it holds that 
\begin{subequations}\label{eq:MVnetworkgeneral}
\begin{align}
    &\d X^\xi_t = \int_{[0,1]\times \Omega} W_t(X^\xi,\mathbb{X}^\alpha(\omega'),\mathbf{w})K(X^\xi_t,\mathbb{X}_t^\alpha(\omega')) \alpha(\xi)(\d \mathbf{w},\d \omega') \d t + \d B_t \\
    &\alpha(\xi)(\d \mathbf{w}, \d\omega') = \eta^{\xi}(\d \mathbf{w},\d\xi')\otimes f_0(\d x')\otimes \mathscr{W}(\d \gamma'), \quad \omega' = (\xi',x',\gamma') \in \Omega,
\end{align}
\end{subequations}
where the existence of a deterministic digraph limit $\eta^\xi$ is provided by Lemma~\ref{lem:CVDig}. Here we consider a pointwise kernel $K$ such that $K_t(X,Y) = K(X_t,Y_t)$ for $X,Y\in\C_{\R^d}$ since all the pathwise interactions will be included in the weight function $W$.

\subsection{Constant network}

First we consider the case where the weight function $W$ does not depend on time, i.e $\phi=0$ in \eqref{eq:odew}. Then, the limit process ${X}^\xi_t := \mathbb{X}_t^\alpha(\omega)$ with input data $\omega = (\xi,X_0,B)$ is characterized by the McKean-Vlasov SDE 
\[\dd X^\xi_t = \int_{\R^d} K(X^\xi_t,y)\nu^\xi_t(\dd y) \d t + \dd B_t,\]
where $\nu^\xi_t ( \d y) := \int_{[0,1] \times \Omega} \mathbf{w} \delta_{\mathbb{X}^\alpha_t (\omega)} ( \d y ) \alpha(\xi) (\d \mathbf{w}, \d \omega)$. 
Next, we want to characterize the conditional law $\rho_t^\xi$ of $X^\xi_t$ knowing $\xi$, which is defined against any test function $\varphi \in \C_b(\R^d)$ by
\[\langle \rho^\xi_t ,\varphi\rangle := \E [\varphi(\mathbb{X}^\alpha_t(\omega))\vert\xi ] = \int_{\R^d\times \mathcal{C}_{\R^d}} \varphi(\mathbb{X}_t^\alpha(\xi,x,\gamma)) f_0(\dd x)\mathscr{W}(\d \gamma).\]
By It\={o}'s formula, using that $\alpha(\omega) \equiv \alpha(\xi)$,
\[\partial_t \rho^\xi_t (x) = -\nabla \cdot \bigg[ \rho_t^\xi(x) \int_{\R^d} K(x,y)\nu^\xi_t(\dd y) \bigg] + \frac{1}{2} \Delta\rho^\xi_t (x),\]
in the sense of distributions.
In the setting \eqref{eq:LimDig} of Lemma~\ref{lem:CVDig}, we can write 
\[ \langle \nu^\xi_t,\varphi\rangle = \int_{[0,1]\times\Omega}\mathbf{w} \,\varphi(\mathbb{X}_t^\alpha(\omega'))\alpha(\xi)(\d \mathbf{w}, \dd\omega') = \int_I \langle \rho^\zeta_t,\varphi\rangle \overline{\eta}^\xi(\dd\zeta), \]
where
\[\overline{\eta}^\xi(\d\zeta) = \int_{[0,1]} \mathbf{w} \,\eta^\xi(\d \mathbf{w}, \d\zeta)\in \pb(I).\]
This finally leads to the Vlasov digraph equation \cite{ayi_largepopulation_2024,kuehn_meanfield_2024}
\begin{equation} \label{eq:Vlasov}
\partial_t \rho^\xi_t (x) = -\nabla \cdot \bigg[ \rho^\xi_t(x) \int_I\int_{\R^d} K(x,y)\rho^\zeta_t(\dd y)\overline{\eta}^\xi_t(\dd\zeta) \bigg] + \frac{1}{2}\Delta\rho^\xi_t (x),
\end{equation}
in the sense of distributions.

\subsection{Pathwise network}

Secondly, we consider the case where the initial weights $\mathbf{w}^{i,j}_N$ are all equal to the same arbitrary value but can later evolve in time. Thus in this case the weight function $W$ does not depend on its last argument: $W_t(x,x',\mathbf{w}) \equiv W_t (x,x')$ is the solution at time $t$ of the ODE \eqref{eq:odew} with prescribed initial condition. This setting reduces the problem to the characterization of the law of a classical exchangeable but pathwise McKean-Vlasov process, here given by the SDE 
\begin{equation}\label{eq:pathwisenet} \dd X_t = \int_{\mathcal{C}_{\R^d}} W_t(X,Y) K(X_t,Y_t) P(\dd Y) \dd t+ \dd B_t,\end{equation}
where $P \in \ps ( \C_{\R^d})$ denotes the path law of $X$. This kind of pathwise network may model systems where bonds between individuals are created or destroyed depending on both the full history and the local states, for instance in the case of aging or for mechanical links subject to physical constraints as in \cite{degond_continuum_2016}. There is in general no guarantee that a closed PDE formulation can exist for the law of nonlinear fully pathwise process such as \eqref{eq:pathwisenet}. In the following (formal) proposition, we highlight an important case where such formulation exists. The standard Fokker-Planck equation on the law of \eqref{eq:pathwisenet} at a given time is coupled with another equation describing the evolution of the law of the link between a pair of two independent processes. It is inspired by the heuristic approach developed in \cite{degond_continuum_2016} for a biological model of tissues that are made of collagen fibers linked together by dynamic elastic bonds. This setting might also be natural for applications in social sciences where bonds between individuals depend on their past history.

\begin{proposition}[Informal]\label{pro:closed}
Let us assume that the function $\phi$ in \eqref{eq:odew} is of the form  
\[ \phi(x,\tilde{x},\mathbf{w}) = \ell(x,\tilde{x}) + \alpha(x,\tilde{x}) \mathbf{w} \] 
for some $\ell,\alpha:\R^d\times \R^d \to \R$ and that $(X_t,\tilde{X}_t,W_t(X_t,\tilde{X}_t))$ has a smooth positive density $H_t$ with respect to the Lebesgue measure, where $\tilde{X}$ is an independent copy of $X$. Then the law $p_t$ of $X_t$ satisfies the following closed system of PDEs:
\begin{align*}
\partial_t p_t(x) &= -\nabla \cdot\left(\int_{\R^d} K(x,y)h_t(x,y)\dd y\right) + \frac{1}{2}\Delta p_t(x), \\ 
\partial_t h_t(x,\tilde{x}) &=  -\nabla_x\cdot \left(h_t(x,\tilde{x})\frac{\int_{\R^d} K(x,y) h_t(x,y)\dd y}{p_t(x)}\right)\\
&\quad-\nabla_{\tilde{x}}\cdot \left(h_t(x,\tilde{x})\frac{\int_{\R^d} K(y,\tilde{x}) h_t(y,\tilde{x})\dd y}{p_t(\tilde{x})}\right)\\
&\quad+ \ell(x,\tilde{x})p_t(x)p_t(\tilde{x}) + \alpha(x,\tilde{x})h_t(x,\tilde{x}) + \frac{1}{2}\Delta_{x,\tilde{x}} h_t(x,\tilde{x}).
\end{align*}
Moreover, $h_t$ is linked to $H_t$ by $h_t(x,\tilde{x}) := \int_\R \mathbf{w} H_t ( x, \tilde{x},\mathbf{w}) \d \mathbf{w}$.   
\end{proposition}

This statement is informal, as we do not specify the conditions that ensure the existence and smoothness of $H_t$.
These properties should result from the analysis of the PDE system above, which is beyond the scope of this paper.
However, we still include the formal derivation below, because we believe that this closure argument is interesting on its own and may lead to a rigorous proof.

\begin{proof}(formal)
The main idea is to introduce the measure $h_t$ on $\R^d\times \R^d$ defined weakly by 
\[\langle h_t,\varphi \rangle = \iint_{\mathcal{C}_{\R^d}\times\mathcal{C}_{\R^d} } \varphi (X_t,\tilde{X}_t) \mathbf{w}_t(X,\tilde{X}) P(\dd X)P(\dd \tilde{X}) = \E[\varphi(X_t,\tilde{X}_t)\mathbf{w}_t(X,\tilde{X})],\]
where $\tilde{X}$ will generically denote an independent copy of $X$. Then It\=o's formula gives the equation on $p_t$. The main question is then to find an equation on $h$.

We have the relations
\[\int_\R H_t(x,\tilde{x},\mathbf{w})\dd \mathbf{w} = p_t(x)p_t(\tilde{x}),\quad \int_\R \mathbf{w} H_t(x,\tilde{x},\mathbf{w})\dd \mathbf{w} = h_t(x,\tilde{x}).\]
Using It\=o's formula, $H_t$ satisfies a simple equation.
Indeed, we take a test function $\varphi \equiv\varphi(x,\tilde{x},\mathbf{w})$ on $\R^d\times \R^d\times\R$ and a couple $(X,\tilde{X})\sim P^{\otimes 2}$ and we compute, writing $\mathbf{w}_t \equiv W_t(X,\tilde{X})$, 
\begin{align*}
    \dd \varphi(X_t,\tilde{X}_t,\mathbf{w}_t) =& \nabla_{x}\varphi(X_t,\tilde{X}_t,\mathbf{w}_t)\cdot \int_{\mathcal{C}_{\R^d}} K(X_t,Y_t)W_t(X,Y) P(\dd Y) \d t \\
    & + \nabla_{\tilde{x}}\varphi(X_t,\tilde{X}_t,\mathbf{w}_t)\cdot \int_{\mathcal{C}_{\R^d}} K(Y_t,\tilde{X}_t)W_t(Y,\tilde{X}) P(\dd Y) \d t \\ 
    &+\partial_\mathbf{w} \varphi(X_t,\tilde{X}_t,\mathbf{w}_t) \phi(X_t,\tilde{X}_t,\mathbf{w}_t) \d t  + \frac{1}{2}\Delta_{x,\tilde{x}}\varphi(X_t,\tilde{X}_t,\mathbf{w}_t) \d t + \d M_t,
\end{align*}
where $M_t$ is a martingale term. Then one can notice that 
\begin{align*}
\int_{\mathcal{C}_{\R^d}} K(X_t,Y_t)W_t(X,Y) P(\dd Y) &= \E_Y [K(X_t,Y_t)W_t(X,Y) \vert X_t ] \\
&= \E[K(X_t,\tilde{X}_t)\mathbf{w}_t|X_t] \\
&= \frac{\int_{\R^d} K(X_t,y) \mathbf{w}' H(X_t,y,\mathbf{w}')\dd y \dd \mathbf{w}'}{\int_{\R^d} H_t(X_t,y,\mathbf{w}')\dd y\dd \mathbf{w}'} \\
&= \frac{\int_{\R^d} K(X_t,y) h(X_t,y)\dd y}{p_t(X_t)}.
\end{align*}
Thus taking the expectation and integrating by part It\=o's formula leads to 
\begin{align*}
\partial_t H(x,\tilde{x},\mathbf{w}) =& -\nabla_x\cdot \left(H(x,\tilde{x},\mathbf{w})\frac{\int_{\R^d} K(x,y) h(x,y)\dd y}{p_t(x)}\right)\\
&-\nabla_{\tilde{x}}\cdot \left(H(x,\tilde{x},\mathbf{w})\frac{\int_{\R^d} K(y,\tilde{x}) h(y,\tilde{x})\dd y}{p_t(\tilde{x})}\right)\\
&-\partial_\mathbf{w}\big(H(x,\tilde{x},\mathbf{w})\phi(x,\tilde{x},\mathbf{w})\big) + \frac{1}{2}\Delta_{x,\tilde{x}} H(x,\tilde{x},\mathbf{w}).
\end{align*}
The equation on $h$ is simply obtained by multiplying by $w$ and integrating on~$\R$:
\begin{align*}
    \partial_t h(x,\tilde{x}) = & -\nabla_x\cdot \left(h(x,\tilde{x})\frac{\int_{\R^d} K(x,y) h(x,y)\dd y}{p_t(x)}\right)\\
&-\nabla_{\tilde{x}}\cdot \left(h(x,\tilde{x})\frac{\int_{\R^d} K(y,\tilde{x}) h(y,\tilde{x})\dd y}{p_t(\tilde{x})}\right)\\
&+\int_\R H(x,\tilde{x},w)\phi(x,\tilde{x},\mathbf{w})\dd \mathbf{w} + \frac{1}{2}\Delta_{x,\tilde{x}} h(x,\tilde{x}).
\end{align*}
This equation is closed in the decoupled linear case: $\phi(x,\tilde{x},\mathbf{w}) = \ell(x,\tilde{x}) + \alpha(x,\tilde{x}) \mathbf{w}$ thanks to the moment properties satisfied by $H$, giving the result.
\end{proof}

In this result, one may for instance think of the variable $x$ as a position, the function $\alpha\equiv\alpha(x,\tilde{x})<0$ as a destruction rate and $\ell$ as a creation term which may both depend on the positions of the end points $x,\tilde{x}$ of a link.

\begin{remark}
    In the degenerate case $\ell\equiv0$ and $\alpha(x,\tilde{x}) = 0$ when $(x,\tilde{x})\in \mathcal{D} = \{|x-\tilde{x}|\leq R\}$ and $\alpha(x,\tilde{x}) = -\infty$ on the boundary, a pre-existing link is immediately destroyed when its length $|x-\tilde{x}|$ exceeds a certain thresholding value $R$. One can expect the same equation on $h$ to hold, without the linear term in $\alpha$ but supplemented with the boundary condition $h|_{\partial\mathcal{D}} = 0$. This can be also be seen by noticing that $h$ is the law of a killed diffusion process. However, this case does not satisfy any Lipschitz assumptions so it remains a formal observation at this stage.
\end{remark}

\subsection{Adaptive network} \label{ssec:Adaptive}

To conclude, we combine the results of the two preceding sections to derive a closed characterization of the limiting law when both an initial network and pathwise interactions are considered. The main model of interest is a linear model for the weight evolution (as in the previous section), corresponding to \eqref{eq:odew} with 
\begin{equation}\label{eq:Kuramoto}\phi(x,\tilde{x},\mathbf{w}) = -\lambda \mathbf{w} + \ell(x,\tilde{x}),\end{equation}
for $\lambda>0$ and some function $\ell : \R^d\times \R^d\to \R$. Such model may be used to model neuron activity and is often seen as a generalization of the Kuramoto model for coupled oscillators on graph sequences \cite{kaliuzhnyi-verbovetskyi_mean_2018,medvedev_continuum_2019} with an additional time evolution of the weights. We refer to \cite{gkogkas_mean_2023,ayi_largepopulation_2024} and the references therein for more details. 

The main observation is that thanks to Duhamel's formula, \eqref{eq:odew} can be rewritten
\[\mathbf{w}_t = \e^{-\lambda t} \mathbf{w}' + \int_0^t \e^{-\lambda(t-s)} \ell(\X_s(\omega),\X_s(\omega'))\d s.\]
Then, reporting this in \eqref{eq:MVnetworkgeneral} leads to 
\begin{multline}\label{eq:MVnetworkduhamel}
    \d X^\xi_t = \e^{-\lambda t}\int_{\R^d} K(X^\xi_t,y)\nu_t^\xi(\d y) \d t \\ + \int_\Omega W^0_t(X^\xi,\X^\alpha(\omega'))K(X^\xi_t,\X^\alpha_t(\omega'))\hat{\alpha}(\xi)(\d\omega') \d t + \d B_t,
\end{multline}
where for $X,Y\in\C_{\R^d}$ and $\xi\in I$
\[W^0_t(X,Y) := \int_0^t \e^{-\lambda(t-s)}\ell(X_s,Y_s)\d s,\quad \hat{\alpha}(\xi)(\d \omega') := \int_{\mathbf{w}
\in[0,1]} \alpha(\xi)(\d \mathbf{w},\d\omega').\]

The following statement provides an answer to \cite[Remark 4.3(ii)]{gkogkas_mean_2023} on the existence of a PDE characterization for the limit of such adaptive networks, see also \cite[Section 3]{ayi_largepopulation_2024}.

\begin{corollary}
Under \eqref{eq:Kuramoto} and the assumptions of Proposition \ref{pro:closed}, the law $\rho_t^\xi$ of $X_t$ knowing $\xi$ satisfies the coupled Vlasov system
\begin{align*}
\partial_t\rho_t^\xi(x) =& - \nabla\cdot \left(\rho_t^\xi(x)\e^{-\lambda t}\int_{\R^d} K(x,y)\rho_t^\zeta(\dd y)\eta_0^\xi(\dd\zeta)\right)\\
& \qquad- \nabla\cdot\left(\int_{\R^d}\int_{\R^d} K(x,y)h^{\xi,\zeta}_t(x,y)\dd y\dd\zeta\right) + \frac{1}{2}\Delta\rho_t^\xi, \\
\partial_t h_t^{\xi,\tilde{\xi}}(x,\tilde{x}) = & -\nabla_x\cdot\left(h_t^{\xi,\tilde{\xi}}(x,\tilde{x}) \int_{\R^d} \e^{-\lambda t}K(x,y) \rho^\zeta_t(\dd y)\eta_0^\xi(\dd \zeta)\right) \\ 
& \qquad\qquad - \nabla_x \cdot\left( h_t^{\xi,\tilde{\xi}}(x,\tilde{x}) \frac{\int_I\int_{\R^d} K(x,y)h_t^{\xi,\zeta}(x,y)\dd y\dd \zeta}{\rho_t^\xi(x)}\right)\\
& -\nabla_{\tilde{x}}\cdot\left(h_t^{\xi,\tilde{\xi}}(x,\tilde{x}) \int_{\R^d} \e^{-\lambda t}K(y,\tilde{x}) \rho^\zeta_t(\dd y)\eta_0^{\tilde{\xi}}(\dd \zeta)\right) \\ 
& \qquad\qquad - \nabla_{\tilde{x}} \cdot\left( h_t^{\xi,\tilde{\xi}}(x,\tilde{x}) \frac{\int_I\int_{\R^d} K(y,\tilde{x})h_t^{\zeta,\tilde{\xi}}(y,\tilde{x})\dd y\dd\zeta}{\rho_t^{\tilde{\xi}}(\tilde{x})}\right) \\ 
& + \rho_t^\xi(x)\rho_t^{\tilde{\xi}}(\tilde{x})\ell(x,\tilde{x}) + \frac{1}{2}\Delta_{x,\tilde{x}}h_t^{\xi,\tilde{\xi}}
\end{align*}

\end{corollary}

\begin{proof}
    Equation~\eqref{eq:MVnetworkduhamel} splits the SDE satisfied by $X^\xi_t$ into two terms, the first one corresponding to a constant network (up to a factor $\e^{-\lambda t}$) and the second one to a pathwise network with function $\phi(x,\tilde{x},w) = \ell(x,\tilde{x})$. By linearity of It\=o's formula, the result is then obtained as the sum of \eqref{eq:Vlasov} and the system in Proposition~\ref{pro:closed} with the only difference that one must keep track of the labels in the pathwise network part. 
\end{proof}

Note that as $t\to+\infty$, the influence of the initial network disappears and the system behaves as the pathwise network of Proposition \eqref{pro:closed}.

\appendix

\section{Useful results on Large Deviations}\label{app:ldp}

\begin{definition}[Large Deviation Principle \cite{dembo_large_2009}] Given a lower-semicontinuous function $I : E \rightarrow [0,+\infty]$ on a topological space $E$, a sequence $( \mu_N )_{N \geq 1}$ in $\mathcal{P} ( E )$ satisfies the Large Deviation Principle (LDP) with rate function $I$ when for any Borel set $A\subset E$, it holds that
\[- \inf_{\mathring{A}} I \leq \liminf_{N \rightarrow \infty} N^{-1} \log \mu_N ( A ) \leq \limsup_{N \rightarrow \infty} N^{-1} \log \mu_N ( A ) \leq - \inf_{\overline{A}} I,\]
where $\mathring{A}$ and $\overline{A}$ denote respectively the interior and closure of $A$. The rate function $I$ is a \emph{good rate function} when all its level sets $\{x\in E,\,\,I(x)\leq\alpha\}$, $\alpha\in[0,+\infty)$, are compact subsets of $E$.
\end{definition}

We recall the definition of the relative entropy functional $(P,Q) \mapsto H (P\vert Q)$ between two probability measures
\begin{equation} \label{eq:defH}
H(P \vert Q) := 
\begin{cases}
    \int \log \frac{\d P}{\d Q} \d P \qquad &\text{if  } P \ll Q, \\
    + \infty \qquad &\text{otherwise.}
\end{cases}
\end{equation}

\begin{theorem}[Sanov]
Let $\mu$ be a probability measure on a Polish space $E$, and let $( X^i )_{i\geq1}$ be a sequence of independent $\mu$-distributed random variables. Then the laws in $\mathcal{P} ( \mathcal{P} ( E ) )$ of the measure-valued random variables $\pi(\vec{X}^N) = \frac{1}{N}\sum_{i=1}^N \delta_{X^i}$ satisfy the large deviation with good rate function given by the relative entropy $\nu \mapsto H ( \nu | \mu )$.
\end{theorem}

See for instance \cite[Theorem 6.2.10]{dembo_large_2009}.

\begin{lemma}[Contraction principle] Let $\mathcal{X}$ and $\mathcal{Y}$ be Hausforff topological spaces and $f:\mathcal{X}\to\mathcal{Y}$ be a continuous function. If a family of probability measures $(\mu_N)_N$ on $\mathcal{X}$ satisfies the LDP with good rate function $I$, then the family of push-forward measures $(f_\#\mu_N)_N$ on $\mathcal{Y}$ satisfies the LDP with good rate function 
\[J : y\in \mathcal{Y} \mapsto J(y) := \inf\{I(x) :\,\,x\in\mathcal{X},\,\,y=f(x)\}.\]
\end{lemma}

See for instance \cite[Theorem 4.2.1]{dembo_large_2009}.

\begin{lemma}[Marginal decomposition] Let $\Sigma = \Sigma_1\times\Sigma_2$ where $\Sigma_1,\Sigma_2$ are Polish spaces and let $\pi:\Sigma\to\Sigma_1,(\sigma_1,\sigma_2)\mapsto\sigma_1$ be the projection on the first component.  For a given measure $\mu\in \pb(\Sigma)$, its restriction to $\Sigma_1$ is denoted by $\mu_1 = \pi_\#\mu \in \pb(\Sigma_1)$. A regular conditional probability distribution (rcpd) given $\pi$ is a mapping $\sigma_1\in\Sigma_1\mapsto \mu^{\sigma_1} \in \pb(\Sigma)$ such that 
\[\mu^{\sigma_1}(\{\sigma : \pi(\sigma)\ne\sigma_1\}) = 0\]
and such that for any measurable subset $A\subset\Sigma$, the map $\sigma_1\mapsto\mu^{\sigma_1}(A)$ is measurable and 
\[\mu(A) = \int_{\Sigma_1} \mu^{\sigma_1}(A)\mu_1(\d \sigma_1).\]
Note that thanks to the first condition, $\mu^{\sigma_1}(\cdot)$ can be considered as a measure on $\Sigma_2$ with for any measurable set $A_2\subset \Sigma_2$, 
\[\mu^{\sigma_1}(A_2) \equiv \mu^{\sigma_1}(\{(\sigma_1,\sigma_2) : \sigma_2\in A_2\}).\]
Now, let $\mu,\nu\in \pb(\Sigma)$. Then the map 
\[\Sigma_1\to[0,+\infty),\,\,\,\sigma_1\mapsto H\big(\nu^{\sigma_1}(\cdot)|\mu^{\sigma_1}(\cdot)\big),\]
is measurable and 
\[H(\nu|\mu) = H(\nu_1|\mu_1) + \int_{\Sigma_1} H\big(\nu^{\sigma_1}(\cdot)|\mu^{\sigma_1}(\cdot)\big)\nu_1(\d\sigma_1).\]
\end{lemma}

See for instance \cite[Section D.3]{dembo_large_2009}

\begin{lemma}[Contraction for entropy] Let $\psi:\mathcal{X}\to\mathcal{Y}$ be a Borel measurable mapping between Polish spaces. Let $\mu\in\pb(\mathcal{X}) $ and $\nu\in\pb(\mathcal{Y})$, then 
\[H(\nu|\psi_\#\mu) = \inf \{ H(\gamma|\mu) : \gamma\in\pb(\mathcal{X}),\,\,\psi_\#\gamma = \nu),\]
where $\inf\emptyset=\infty$ by convention. 
\end{lemma}

See for instance \cite[Lemma A.1]{fischer_form_2014}.

\section{Labeled Sanov theorem} \label{sec:app}

This appendix is devoted to the proof of a variation on Sanov's classical theorem in large deviation theory that is well-suited for handling digraph measures.

Let $F$ be a Banach space.
Let $( \mathbf{w}^{i,j}_N )_{1 \leq i, j \leq N} \subset [0,1]$ and $( Y^{1,j} )_{1 \leq j \leq N} \subset F$, $N \geq 1$,  be arrays of random variables. We define the 
(random) digraph measure $\alpha_N \in L^\infty ( [0,1], \ps ([0,1]\times I  \times F) )$ by
\[ \alpha_N (\xi) := \frac{1}{N} \sum_{j = 1}^N \delta_{(\mathbf{w}_N^{[\xi]_N,j},j/N,Y^{j,N})}. \]
We consider the spaces 
\[ \A := L^\infty ( [0,1] , \ps ( [0,1]\times I \times F ) ), \qquad \B := L^1 ( [0,1] , \mathrm{BL} ( [0,1]\times I \times F ) ), \]
where $\mathrm{BL} ( [0,1]\times I \times F )$ denotes the Banach space of real-valued bounded-Lipschitz functions endowed with the norm $\lVert \cdot \rVert_\infty + \lVert \cdot \rVert_{\mathrm{Lip}}$.
We notice that $\A$ is included in the topological dual of the Banach space $( \B, \d_{\infty, \mathrm{BL}})$.

\begin{theorem}[Labeled Sanov] \label{thm:LabSanov}
Assume that $(( \mathbf{w}^{i,j}_N )_{1 \leq i \leq N}, Y^{j,N} )_{1 \leq j \leq N}$ are independent and that for some measurable space $\mathcal{U}$, there exist $W \in \B ( I^2 \times \mathcal{U}, [0,1] )$ and $\mu \in \B (I, \ps ( \mathcal{U} \times F))$ such that 
\begin{multline} \label{eq:EXpCV}
\forall \phi \in \B, \quad \frac{1}{N} \sum_{j=1}^N \log \E \exp \bigg[ \int_I \phi ( \xi, \mathbf{w}^{[\xi]_N,j}_N , j/N, Y^{j,N} ) \d \xi \bigg] \xrightarrow[N \rightarrow + \infty]{} \\
\int_I \log \bigg\{ \int_{{\mathcal{U} \times F}} \exp \bigg[ \int_I \phi ( \xi, W ( \xi, \xi', u ) , \xi', y ) \d \xi \bigg] \mu^{\xi'}(\d u,  \d y) \bigg\}\d \xi'. 
\end{multline} 
Then $( \mathrm{Law} ( \alpha_N ) )_{N  \geq 1}$ satisfies the LDP in $( L^\infty ( [0,1] , \ps ( [0,1]\times I \times F ) ), d_{\infty,\mathrm{BL}} )$ with good rate function
\[ I ( \alpha ) := 
\begin{cases}
H ( \nu \vert \d \xi' \mu^{\xi'} ), \quad &\text{if  } \alpha(\xi) ( \d \xi', \d \mathbf{w}, \d y ) = \int_{\mathcal{U}} \delta_{W(\xi,\xi',u)} ( \d \mathbf{w}) \nu( \d \xi', \d u, \d y), \\
+ \infty \quad &\text{otherwise}.
\end{cases}
\]
As a consequence, the sequence $(\alpha_N)_{N \geq 1}$ almost surely converges in $\B$ towards $\overline\alpha : \xi \mapsto \int_{\mathcal{U}} \delta_{W(\xi,\xi',u)} ( \d \mathbf{w})  \d \xi' \mu^{\xi'}(\d u, \d y')$.
\end{theorem}

If the $( \mathbf{w}^{i,j}_N )_{1 \leq i,j \leq N}$ are i.i.d., and $Y^{j,N} = Y^j$ for an i.i.d. sequence $ (Y^{j} )_{j \geq 1}$, then we notice that the convergence \eqref{eq:EXpCV} reduces to the Riemann sum convergence theorem.

\begin{proof}
We adapt the proof of the Sanov theorem from \cite[Theorem 6.2.10]{dembo_large_2009} in the spirit of the Gärtner-Ellis approach.
This proof is based on \cite[Corollary 4.6.11-(a)]{dembo_large_2009}.

We consider $( \alpha_N )_{N \geq 1}$ as a sequence taking values in the topological dual $\mathcal{X}$ of $( \B, \d_{\infty, \mathrm{BL}})$.
This choice satisfies \cite[Assumption 4.6.8]{dembo_large_2009} for $\mathcal{W} = \B$.
From the independence of $(( \mathbf{w}^{i,j}_N )_{1 \leq i \leq N}, Y^{j,N} )_{1 \leq j \leq N}$ and \eqref{eq:EXpCV}, we get that
\[ \forall \phi \in \B, \quad \frac{1}{N} \log \E \exp \bigg[ N \int_{[0,1]\times I \times F} \phi ( \xi, \mathbf{w},\xi',y) \alpha_N ( \xi) ( \d \mathbf{w} , \d \xi',\d y ) \d \xi \bigg] \xrightarrow[N \rightarrow + \infty]{} \Lambda ( \phi ), \]
where $\Lambda ( \phi )$ denotes the right-hand side of \eqref{eq:EXpCV}.
This quantity is finite everywhere in $\B$. 
Moreover, for every $\varphi, \psi \in \B$, $g : t \mapsto \Lambda( \phi + t \psi )$ is differentiable at $t=0$ with
\[ g'(0) = \int_I \frac{\int \int_I \psi ( \xi,W(\xi,\xi,u),\xi',y) \d \xi \exp [ \int_0^1 \phi ( \xi,W(\xi,\xi,u),\xi',y) \d \xi ] \mu^{\xi'}(\d u,\d y)}{\int \exp [ \int_I \phi ( \xi,W(\xi,\xi,u),\xi',y) \d \xi ] \mu^{\xi'}(\d u,\d y)} \d \xi', \]
the dominations being straightforward from $\phi,\psi \in \B$.
$\Lambda$ thus being Gateaux-differentiable, all the conditions of \cite[Corollary 4.6.11-(a)]{dembo_large_2009} are satisfied, so that $( \mathrm{Law} ( \alpha_N ) )_{N  \geq 1}$ satisfies the LDP with good convex rate function
\[ \Lambda^\star (\omega) := \sup_{\phi \in \B} \, \langle \omega, \phi \rangle - \Lambda ( \phi ). \]
Adapting \cite[Lemma 6.2.13]{dembo_large_2009},
let us now show that $\Lambda^\star = I$, when extending $I$ to $\mathcal{X}$ by $I(\omega) = +\infty$ if $\omega \notin \A$.
As a consequence of Fatou's lemma, the convex function $I$ is lower semi-continuous on $\mathcal{X}$ -- this is a slight variation on \cite[Lemma 6.2.16]{dembo_large_2009}.
Moreover, $\mathcal{X}$ endowed with the $\sigma^\star$-topology induced by $\B$ is a locally convex Hausdorff space, whose topological dual is $\B$ from \cite[Theorem B.8]{dembo_large_2009}.
From \cite[Lemma 4.5.8]{dembo_large_2009}, it is now sufficient to show that for every $\phi \in \B$,
\begin{equation} \label{eq:AbsLegendre}
\Lambda (\phi) = \sup_{\omega \in \mathcal{X}} \, \langle \omega, \phi \rangle - I (\omega).
\end{equation} 
From the expression of $I$, we can restrict the minimization to those $\omega \in \mathcal{X}$ such that
\[ \omega (\xi) = \int_{\mathcal{U}} \delta_{W(\xi,\xi',u)} ( \d \mathbf{w}) f(\xi',u,y) \d \xi'\mu^{\xi'}(\d u, \d y), \]
for some $f$ which is a probability density for the measure $\d\xi'\mu^{\xi'}$ on $I\times\mathcal{U}\times F$. Then, for any $\omega\in\mathcal{X}$ and $\phi\in\mathcal{B}$, one can write 
\begin{align*}
    &\langle \omega,\phi\rangle - I(\omega) =\\ &\int_{I\times \mathcal{U}\times F} f(\xi',u,y)\left(\int_I \phi(\xi,W(\xi,\xi',u),\xi',y)\d\xi - \log f(\xi',u,y)\right )\d\xi'\mu^{\xi'}(\d u,\d y)
\end{align*}
Considering the function 
\begin{align*}&f(\xi',u,y) = \frac{\exp\left[\int_I \phi(\xi,W(\xi,\xi',u),\xi',y)\d\xi\right]}{\mathcal{Z}},\\ 
&\mathcal{Z}=\int_{I}\int_{\mathcal{U}\times F} \exp\left[\int_I \phi(\xi,W(\xi,\xi',u),\xi',y)\d\xi\right]\mu^{\xi'}(\d u,\d y)\d\xi',\end{align*}
one gets 
\[\langle \omega,\phi\rangle - I(\omega) = \log\mathcal{Z}\geq \Lambda(\phi),\]
where the last inequality comes from Jensen's inequality on the space $I$. For the converse inequality, for any probability density function $f$ on $I\times\mathcal{U}\times F$, we apply Jensen's inequality on the space $\mathcal{U}\times F$ to get 
\begin{align*}
\Lambda (\phi) &\geq \int_I \log \bigg\{ \int_{{\mathcal{U} \times F}} \1_{f >0} f^{-1}(\xi',u,y) \exp \bigg[ \int_I \phi ( \xi, W(\xi,\xi',u) , \xi', y ) \d \xi \bigg] \\
&\qquad\qquad\qquad\qquad\qquad\qquad\qquad\qquad \times f(\xi',u,y) \mu^{\xi'}(\d u,  \d y) \bigg \}\d \xi', \\
&\geq \int_{I\times \mathcal{U}\times F} \1_{f > 0} \log \frac{\exp \big[ \int_I\phi ( \xi, W(\xi,\xi',u) , \xi', y ) \d \xi \big]}{f(\xi',u,y)} f(\xi',u,y) \mu^{\xi'}(\d u,  \d y) \d \xi',\\
&= \langle \omega,\phi\rangle - I(\omega).
\end{align*}

Since $\overline\alpha$ is the unique minimiser of the good rate function $I$, the a.s. convergence then classically follows from the Borel-Cantelli lemma.
\end{proof}

\section*{Acknowledgments}

This work was initiated when AD was affiliated to the Institute for the Advanced Study of Human Biology (ASHBi), Kyoto University. LPC acknowledges the hospitality of ASHBi where part of this work was carried out. The work of AD is partially supported by the KAKENHI Grant-in-Aid for Early-Career Scientists (JP23K13015).

\printbibliography

\end{document}